\documentclass[12pt,reqno]{amsart} 

\usepackage{amssymb,bbm,marginnote,url}
\usepackage{aliascnt}
\usepackage{doi}
\usepackage{pgfplots}

\usepackage[super]{nth}

\usepackage{psfrag}

\usepackage{color}

\newcommand\blue[1]{\textcolor{blue}{#1}}

\newtheorem{theorem}{Theorem}

\newaliascnt{lemma}{theorem}
\newtheorem{lemma}[lemma]{Lemma}
\aliascntresetthe{lemma}

\newaliascnt{proposition}{theorem}
\newtheorem{proposition}[proposition]{Proposition}
\aliascntresetthe{proposition}

\newaliascnt{corollary}{theorem}
\newtheorem{corollary}[corollary]{Corollary}
\aliascntresetthe{corollary}

\newaliascnt{conjecture}{theorem}

\aliascntresetthe{conjecture}

\newaliascnt{openQ}{theorem}

\aliascntresetthe{openQ}

\newaliascnt{quest}{theorem}

\aliascntresetthe{quest}

\newaliascnt{questx}{conjx}

\aliascntresetthe{questx}

\theoremstyle{definition}

\newaliascnt{defn}{theorem}

\aliascntresetthe{defn}

\newaliascnt{example}{theorem}

\aliascntresetthe{example}

\newaliascnt{rem}{theorem}

\aliascntresetthe{rem}

\makeatletter
\def\tagform@#1{\maketag@@@{\ignorespaces#1\unskip\@@italiccorr}}
\let\orgtheequation\theequation
\def\theequation{(\orgtheequation)}
\makeatother
\def\equationautorefname~{}

\renewcommand{\Im}{\operatorname{Im}}

\newcommand{\D}{\mathbb{D}}
\newcommand{\CD}{\overline{\mathbb{D}}}

\newcommand{\e}{\varepsilon}
\newcommand{\R}{{\mathbb R}}
\newcommand{\Rn}{{\mathbb R}^n}

\newcommand{\Sp}{{\mathbb S}}
\newcommand{\C}{{\mathbb C}}
\newcommand{\HH}{{\mathbb H}}

\let\oldmarginnote\marginnote
\renewcommand{\marginnote}[1]{\oldmarginnote{\tiny \blue{#1}}}

\begin{document}
	\title{Robin spectrum: two disks maximize the third eigenvalue}


	\begin{abstract}
		The third eigenvalue of the Robin Laplacian on a simply-connected planar domain of given area is bounded above by the third eigenvalue of a disjoint union of two disks, provided the Robin parameter lies in a certain range and is scaled in each case by the length of the boundary. Equality is achieved when the domain degenerates suitably to the two disks.
	\end{abstract}
\author[]{A. Girouard and R. S. Laugesen}
\address{Department de Math\'ematiques et Statistique, Univ.\ Laval, Quebec, Qc, Canada}
\email{Alexandre.Girouard@mat.ulaval.ca}
\address{Department of Mathematics, University of Illinois, Urbana,
	IL 61801, U.S.A.}
\email{Laugesen\@@illinois.edu}

	\maketitle
	
	
	\section{\bf Introduction} \label{sec:intro}
	
What shape of drum-head gives the largest second overtone? The shape optimization problem is to maximize the third eigenvalue of the Laplacian under suitable geometric constraints and boundary conditions. 

First we formulate the problem, and then state the sharp upper bound on the  eigenvalue. For a bounded domain $\Omega\subset\R^2$ with Lipschitz boundary, the Robin eigenvalue problem with parameter $\alpha\in\R$ is to find all numbers $\lambda\in\R$ for which a nonzero function $u:\overline{\Omega}\rightarrow\R$ exists satisfying
\[
\begin{split}
\Delta u + \lambda u & = 0 \quad \text{in $\Omega$,} \\
\partial_\nu u + \alpha u & = 0 \quad \text{on $\partial \Omega$,} 
\end{split}
\]
	where $\partial_\nu u$ is the normal derivative of $u$ in the outward direction.
The eigenvalues form an unbounded sequence
	$$
	\lambda_1(\Omega;\alpha)\leq\lambda_2(\Omega;\alpha)\leq\lambda_3(\Omega;\alpha)\leq\cdots\nearrow\infty,
	$$
	where each one is repeated according to its multiplicity.
	The corresponding Rayleigh quotient is
	$$Q_\alpha(u) = \frac{\int_{\Omega}|\nabla u|^2\,dA+\alpha\int_{\partial\Omega}|u|^2\,ds}{\int_{\Omega}|u|^2\,dA}, \qquad u \in H^1(\Omega;\C) .$$
From the Rayleigh quotient, the spectrum is easily seen to be scale invariant when the eigenvalues are normalized by area and the Robin parameter is scaled by boundary length; that is, 
\[
\text{$\lambda_j(\Omega;\alpha/L)A$ \hspace*{4pt} is scale invariant}
\]
where $L=$length of $\partial \Omega$ and $A=$area of $\Omega$.
Scale invariance means the expression takes the same value for all dilations of $\Omega$.
	
The normalized first eigenvalue $\lambda_1(\Omega; \alpha/L)A$ is maximal for a degenerate rectangle whenever $\alpha \in \R$. The second eigenvalue $\lambda_2(\Omega; \alpha/L)A$ is maximal among simply-connected domains for the disk whenever $\alpha\in [-2\pi,2\pi]$, as shown by Freitas and Laugesen \cite[Theorems A,B]{FreiLauS2018}.

This paper proves an optimal upper bound on the normalized third eigenvalue $\lambda_3(\Omega;\alpha/L)A$ among simply-connected planar domains. The upper bound is attained in a suitable limit of simply-connected domains  degenerating to a disjoint union of two disks.
	
\begin{theorem}[Third Robin eigenvalue is maximal for the double disk] \label{thm:main}
Fix $\alpha\in[-4\pi,0]$.
If $\Omega\subset\R^2$ is a simply-connected bounded Lipschitz domain whose boundary is a Jordan curve then
\begin{gather}\label{ineq:main}
  \lambda_3(\Omega;\alpha/L) A <\lambda_3(\D\sqcup\D;\alpha/4\pi) 2\pi.
\end{gather}
Furthermore, equality is attained asymptotically for the domain $\Omega_\e = (\D - 1 + \e)\cup (\D + 1 - \e)$ that approaches a double disk as $\e \to 0$.
\end{theorem}
The third eigenvalue of the disjoint union $\D\sqcup\D$ is simply the second eigenvalue of one of the disks, and so the theorem says $\lambda_3(\Omega;\alpha/L) A < \lambda_2(\D;\alpha/4\pi) 2\pi$. This disk eigenvalue can be computed explicitly in terms of Bessel functions, as explained in \autoref{sec:diskproblem}. 

To rephrase conclusion \eqref{ineq:main} another way, write $\Omega^{\star\star}$ for the union of two disjoint disks each having half the area of $\Omega$. Then by scale invariance, the inequality is equivalent to
  $$\lambda_3(\Omega;\alpha/L(\Omega))<\lambda_3(\Omega^{\star\star};\alpha/L(\Omega^{\star\star})).$$

The Neumann case of the theorem ($\alpha=0$) is a result of Girouard, Nadirashvili and Polterovich \cite{GirNadPolt2009}. Their result was generalized by Bucur and Henrot \cite{BucurHenrot2019} to arbitrary domains in higher dimensions.

We do not know whether the range $\alpha \in [-4\pi,0]$ in the theorem can be enlarged. The proof holds unchanged when $\alpha \in (0,4\pi]$ except the uniqueness and continuous dependence proof for the normalizing point $w$ in \autoref{centerofmass} breaks down because the excited Robin state has nonmonotonic radial part when $\alpha > 0$; hence the trial function orthogonality in \autoref{propo:existenceCap} is not known when $\alpha>0$. 
Note the theorem definitely fails in the Dirichlet limit $\alpha\to\infty$, since the Dirichlet eigenvalues of domains of given area can be made arbitrarily large by taking long, thin domains.

Perimeter scaling on the Robin parameter is essential to \autoref{thm:main}. Without it, the double disk is not the maximizer for $\lambda_3(\Omega;\alpha)A$ when $\alpha<0$, according to numerical work by Antunes, Freitas and Krej\v{c}i\v{r}\'{\i}k \cite[Figure 4]{AntunesFreitasKrejcirik2017}. 

It is an open problem to extend \autoref{thm:main} to higher dimensions. Indeed, it is already an open problem to extend Freitas and Laugesen's result on the second eigenvalue $\lambda_2(\Omega; \alpha/L)A$. Conformal mappings as used in their paper and this one are not available in higher dimensions, and so a different kind of proof would be be needed.

\autoref{thm:main} implies a sharp upper bound on the second positive Steklov eigenvalue.
Write $0=\sigma_0 < \sigma_1 \leq \sigma_2 \leq \dots$ for the Steklov eigenvalues of $\Omega\subset\R^2$, which correspond to the eigenvalue problem
\[
\begin{split}
\Delta u & = 0 \ \ \quad \text{in $\Omega$,} \\
\partial_\nu u & = \sigma u \quad \text{on $\partial \Omega$.} 
\end{split}
\]
Notice the product $\sigma_j L$ is scale invariant.
	\begin{corollary}[Sharp bound on the second nonzero Steklov eigenvalue] \label{cor:steklov}
	  If $\Omega\subset\R^2$ is a simply-connected bounded Lipschitz domain whose boundary is a Jordan curve then $$\sigma_2(\Omega) L(\Omega)<4\pi .$$ Equality is attained asymptotically for $\Omega_\e = (\D - 1 + \e)\cup (\D + 1 - \e)$ as $\e \to 0$.
	\end{corollary}
The non-strict inequality $\sigma_2 L \leq 4\pi$ was proved directly by Hersch, Payne and Schiffer \cite{HPS1975}. They further found $\sigma_k L \leq 2\pi k$ for each $k$. Strict inequality was obtained for $k=2$ by Girouard and Polterovich \cite{GirPolt2010}, who established asymptotic sharpness as the domain degenerates suitably to a union of $2$ disjoint disks. 

Might our \autoref{thm:main} for the third Robin eigenvalue extend to the $k$-th Robin eigenvalue being maximal at the union of $k$ disjoint disks, for all $k \geq 2$ and appropriate values of $\alpha$?
Any such generalization will not be straightforward, because when $k=4$ the conjecture fails already at $\alpha=0$ by numerical work of Antunes and Freitas \cite[Figure 1]{AntunesFreitas2012}. Their computations reveal that the fourth Neumann eigenvalue (the third nonzero one) seems to be maximal not for the union of three disjoint disks but for something close to a 3-fold rotationally symmetric overlapping union of three disks.

\subsection*{What is new in this paper?}
The strategy of the present paper is to combine conformal techniques in $2$ dimensions from Girouard, Nadirashvili and Polterovich \cite{GirNadPolt2009}, and particularly their parameterized family of hyperbolic caps, with trial function insights from Bucur and Henrot \cite{BucurHenrot2019}. Both these papers are concerned with maximizing the third Neumann eigenvalue. For the third Robin eigenvalue we must additionally handle a boundary term in the Rayleigh quotient, and so we incorporate the perimeter-scaling ideas of Freitas and Laugesen \cite{FreiLauS2018} from their work maximizing the second Robin eigenvalue.

The current paper provides certain simplifications in comparison to \cite{GirNadPolt2009}, even for the original case of Neumann eigenvalues. Rather than finding a 2-dimensional space of real-valued trial functions that satisfy one orthogonality condition and an additional ``inertia relation'', as in that paper, here we find a single complex-valued trial function satisfying two orthogonality conditions.
Also, the topological argument is simpler than in \cite{GirNadPolt2009}. We hope these improvements make it easier to generalize the approach to other situations. 
		
Finally, the ``pulling apart with a weight'' argument in \autoref{sec:saturation} by which we prove saturation in the main theorem is different and simpler than earlier approaches in the Neumann and Steklov cases for approaching the disjoint union of disks. 

\subsection*{Literature on upper bounds for eigenvalues of the Laplacian}
The question of maximizing individual eigenvalues of the Laplacian goes back at least to work of Szeg\H{o} \cite{Szego1954}. He proved that among simply-connected planar domains of prescribed area, the first nonzero  Neumann eigenvalue $\mu_2(\Omega)$ is largest for the disk, and only the disk. Weinberger \cite{Weinberger1956} generalized the result to all domains in all dimensions. 
Weinstock \cite{Weinstock1954} soon discovered a modification of Szeg\H{o}'s argument that led to the sharp upper bound $\sigma_1 L \leq 2\pi$ on the first nonzero eigenvalue of the Steklov problem, this time under perimeter constraint. The disk is again the unique maximizer. These Neumann and Steklov results were recently extended to the  Robin Laplacian by Freitas and Laugesen \cite{FreiLauW2018,FreiLauS2018}, who showed the ball maximizes the second eigenvalue $\lambda_2(\Omega;\alpha)$ when $\alpha$ lies in a certain range and the volume is fixed.

In the context of surfaces without boundary, similar maximization problems were taken up by Hersch. Given a compact smooth surface $S$ equiped with a Riemannian metric $g$, the Laplace--Beltrami operator $\Delta_g$ has discrete unbounded spectrum $0=\lambda_1(S,g)\leq\lambda_2(S,g)\leq\cdots\to+\infty$. Hersch \cite{Hersch1970} proved that for all Riemannian metrics on the sphere $\Sp^2$ with area equal to $4\pi$, the eigenvalue $\lambda_2(\Sp^2,g)$ is less than or equal to $2$, with equality holding when $g$ is the standard ``round'' metric induced from the embedding of the sphere into $\R^3$. On an arbitrary compact orientable surface $S$, Yang and Yau \cite{YangYau} used a conformal branched covering $S\to\Sp^2$ to bound $\lambda_2(S,g)$ in terms of the genus $\gamma\geq 0$ and area $A$ of the surface. Their bound was improved by El Soufi and Ilias \cite{ElSoufiIlias} to 
\begin{gather}\label{ineq:YangYau}
  \lambda_2(S,g)A\leq 8\pi\left\lfloor\frac{\gamma+3}{2}\right\rfloor.
\end{gather}
When $\gamma=0$, one recovers the above sharp result of Hersch for the sphere. Inequality \eqref{ineq:YangYau} is also sharp for $\gamma=2$ by work of Nayatani and Shoda~\cite{NayataniShoda}, who solved a conjecture from~\cite{JLNNP2005}, but the inequality is strict and not sharp for all values of $\gamma\notin\{0,2\}$, according to recent work of Karpukhin \cite{KarpukhinYangYau}. The sharp upper bound $\lambda_2(\mathbb{T}^2,g)A\leq 8\pi^2/\sqrt{3}$ for metrics on the torus was determined by Nadirashvili \cite{NadTorus}. In the non-orientable case, sharp upper bounds for $\lambda_2(S,g)$ are known for the projective plane \cite{LiYau} and Klein bottle \cite{CianciKarpukhinMedvedev,ElSoufiGiacominiJazar}, in the latter case proving a conjecture by Jakobson, Nadirashvili and Polterovich~\cite{JakNadPoltKlein}.
%

Sharp bounds for higher eigenvalues are significantly more difficult to obtain. Nadirashvili~\cite{NadSphere} proved for the sphere that $\lambda_3(\Sp^2,g)A\leq 16\pi$, and he conjectured $\lambda_k(\Sp^2,g)A\leq 8(k-1)\pi$ for all $k\geq 1$. This was proved by him and Sire~\cite{NadSire2017} for $k=4$, and recently for all $k\geq 1$ by Karpukhin, Nadirashvili, Penskoi and Polterovich~\cite{KNPP2019}. The paper~\cite{NadSphere}, while extremely difficult to understand, has been quite influential. In particular, it led to a sharp upper bound on the third Neumann eigenvalue $\mu_3(\Omega)$ among simply-connected planar domains of given area, obtained by Girouard, Nadirashvili and Polterovich \cite{GirNadPolt2009}. Their result was generalized by Bucur and Henrot \cite{BucurHenrot2019} to arbitrary domains in all dimensions. See also Petrides \cite{Petrides2014} for upper bounds on $\lambda_3$ on spheres of arbitrary dimensions.

Returning to the Robin problem on euclidean domains, we recommend a survey article by Bucur, Freitas and Kennedy \cite{BucFreiKen2017}, which provides a good overview of Robin spectral problems and results, including upper and lower bounds and asymptotics as $\alpha \to \pm \infty$. Many more open problems for Robin eigenvalues and their gaps and ratios are stated by Freitas and Laugesen \cite{FreiLauW2018,FreiLauS2018} and Laugesen \cite{Lau2019}.

\subsection*{Plan of the paper}
The next two sections gather tools for our constructions: M\"{o}bius transformations, hyperbolic caps, and conformal maps between those caps and the disk. Then we recall properties of the Robin eigenfunctions on the disk. Trial functions are constructed in \autoref{sec:trialfuncts}, where they are shown to be orthogonal to the first two Robin eigenfunctions on $\Omega$. Strict inequality for \autoref{thm:main} is proved in \autoref{sec:main}, and \autoref{sec:saturation} shows asymptotic sharpness for the union of two disks. The Steklov result \autoref{cor:steklov} is deduced in \autoref{sec:steklov}.

	\subsection*{Notation}
	The unit disk is $\D = \{ z\in\C : |z|<1 \}$, the upper halfplane is $\HH = \{ z\in\C : \Im z > 0 \}$, and the upper halfdisk is $\D^+ = \D \cap \HH$.
	
	The function spaces $L^2(\Omega;\C)$ and Sobolev $H^1(\Omega;\C)$ of complex valued functions will be used, although for the sake of brevity we will generally omit the $\C$ from the notation. 
	
	A conformal map is a conformal diffeomorphism, holomorphic in both directions.

\section{\bf M\"obius maps and hyperbolic caps}
	
Our estimation of $\lambda_3(\Omega;\alpha/L)$ in \autoref{thm:main} will rely on a variational characterization of the third eigenvalue as the minimum of the Rayleigh quotient taken over all trial functions orthogonal to the first two eigenfunctions:
	\begin{align}
	& \lambda_3(\Omega;\alpha/L) \label{eq:varchar} \\ 
	& =\min\left\{Q_{\alpha/L}(u)\,:\,u\in H^1(\Omega) \setminus \{ 0 \}, \int_{\Omega}uf_1\,dA=\int_{\Omega}uf_2\,dA=0 \right\} \notag
	\end{align}
	where the $f_j$ are $L^2$-orthonormal real-valued eigenfunctions corresponding to the eigenvalues $\lambda_j(\Omega;\alpha/L)$. Remember the trial function $u\in H^1(\Omega)$ may be complex-valued.
	
	We will construct a $4$-parameter family of complex-valued trial functions, in order to obtain enough degrees of freedom to get a trial function orthogonal to $f_1$ and $f_2$. Two parameters will come from a family of M\"{o}bius transformations of the disk, and two more from a family of hyperbolic caps inside the disk. 
	
\subsection*{M\"obius transformations}
Given $w \in \overline{\D}$, let
\[
M_w(z) = \frac{z+w}{z\overline{w}+1} , \qquad z \in \D .
\]
Notice that when $w\in\D$, the function $M_w$ is a M\"{o}bius self-map of the disk and its boundary circle, with $M_w(0)=w$ and $M_w(-w)=0$, and 
\[
M_{-w}=M_w^{-1} .
\]
A rotational conjugation or invariance property of these maps is that  
\begin{equation}
M_{pw} = p \circ M_w \circ p^{-1} , \qquad p \in S^1 , \label{eq:Mxiconj}
\end{equation}
as one sees by writing $p=e^{i\theta}$ and evaluating the right side at $z$ as $e^{i\theta} M_w(e^{-i\theta}z)$. Also, $M_{pt}$ fixes the points $\pm p$ since 
\[
M_{pt}(\pm p) = \pm p , \qquad t \in (-1,1) .
\]
When $w\in\partial\D$ the function $M_w$ is constant on the disk, with $M_w(z)=w$ for each $z\in\D$.

\subsection*{Hyperbolic caps }
Let $\gamma$ be a geodesic in the Poincar\'e disk model; that is, either a diameter or the intersection of the disk with a circle that is orthogonal to the boundary $S^1=\partial \D$. The closure in $\R^2$ of each connected component of $\D\setminus\gamma$ is called a hyperbolic cap, as shown in \autoref{fig:capbasic}. The geodesic $\gamma$ is contained in both caps. Its endpoints are called $a$ and $b$. 
\begin{figure}
  \includegraphics[scale=0.15]{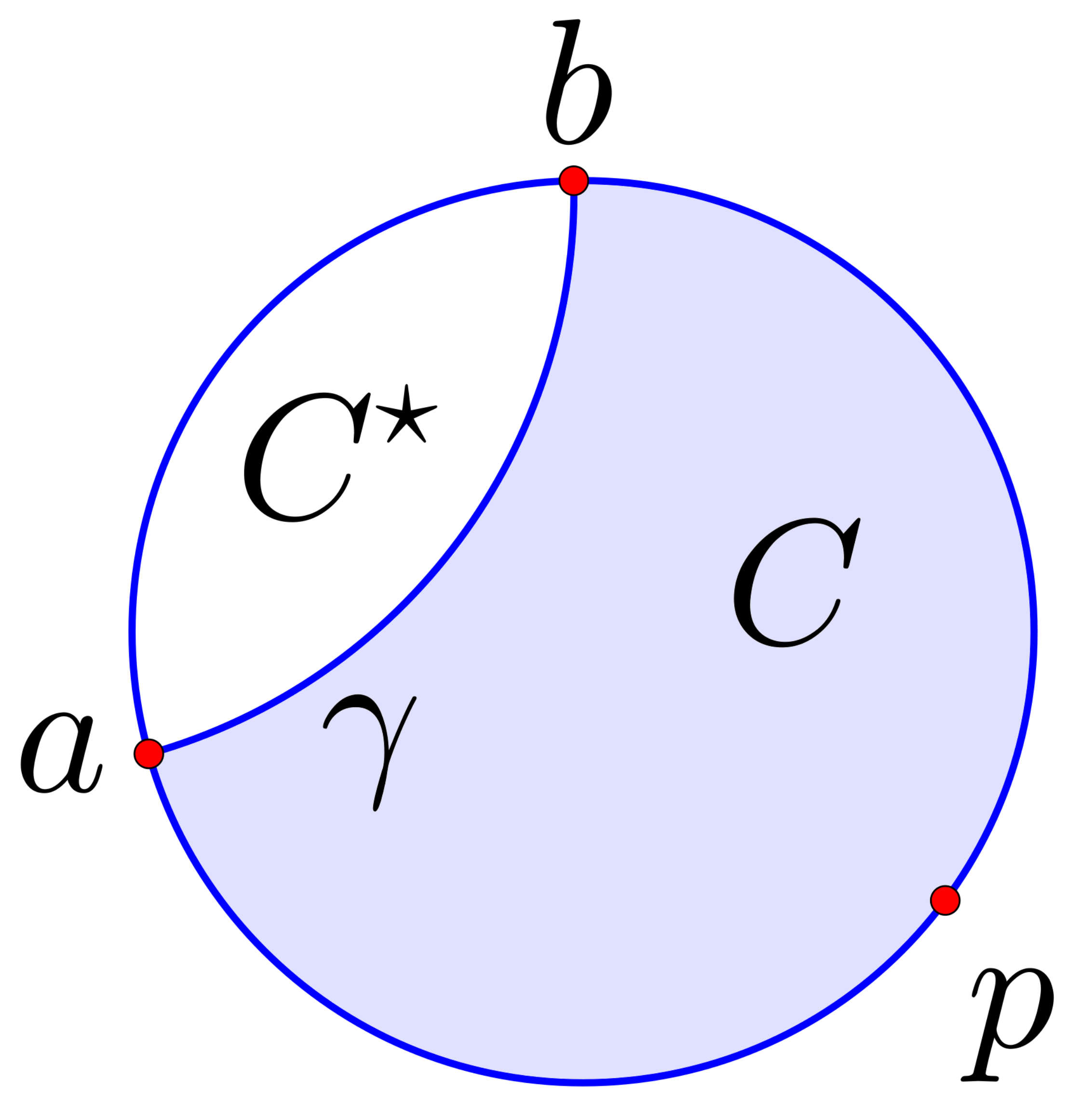}
\caption{\label{fig:capbasic} The hyperbolic caps $C$ and $C^\star$ are the closures of the connected components of $\CD\setminus\gamma$, where $\gamma$ is a geodesic in the Poincar\'e disk model.}
\end{figure}

We want to parameterize the family of caps. The ordered endpoints $a,b \in S^1$ provide a parameterization, and so the family is clearly $2$-dimensional. It turns out to be more convenient to parameterize using the ``center'' $p$ and ``size'' $t$ of the cap $C$, as follows. 

For each point $p\in S^1$, let $C_{p,0}$ be the half-disk ``centered'' at $p$:
$$C_{p,0}=\{z\in\CD\,:\,z\cdot p\geq 0\},$$
where $z$ and $p$ are regarded in this definition as vectors in $\R^2$.
For $t\in (-1,1)$, define
the hyperbolic cap $C=C_{p,t}\subset\CD$ by
\begin{equation} \label{eq:capdef}
C_{p,t}=M_{-pt}(C_{p,0}),
\end{equation}
 as illustrated in \autoref{fig:capdef}. 
The definition is consistent when $t=0$, since $M_0$ is the identity map. The caps are related rotationally in a natural way, according to 
\[
C_{p,t} = p C_{1,t} ,
\]
which is obvious for $t=0$ (half-disks) and can be checked for $t \neq 0$ using the definition of $C_{p,t}$ and the rotational invariance in \eqref{eq:Mxiconj}. The complementary cap is $C^\star=C_{-p,-t}$.  
\begin{figure}
  \includegraphics[scale=0.15]{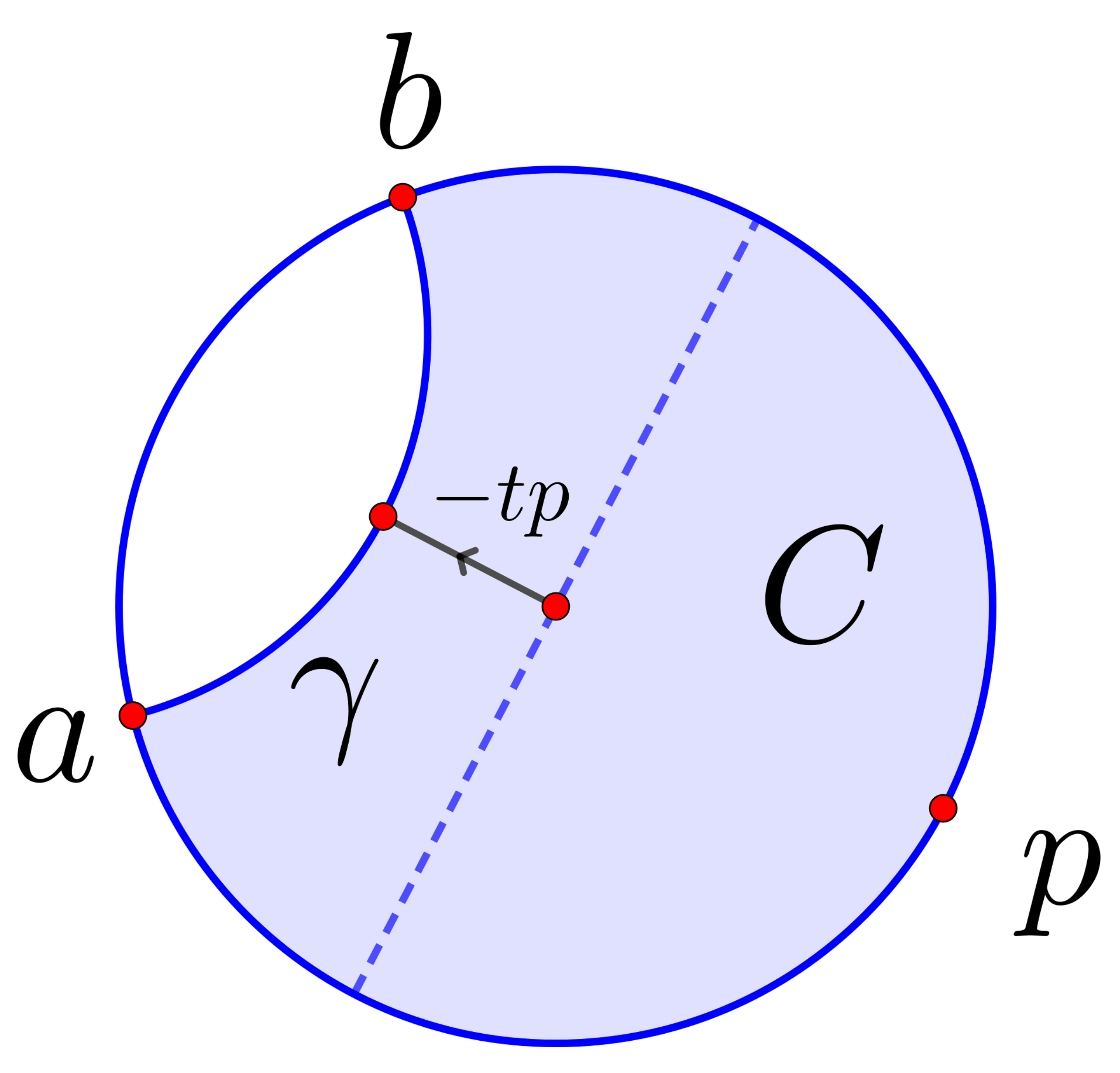}
  \caption{\label{fig:capdef} The hyperbolic cap $C=C_{p,t}$ is the image of the half-disk $C_{p,0}$ under the M\"obius transform $M_{-pt}$. Positive $t$ values correspond to caps larger than a half-disk, as shown here.}
\end{figure}

Importantly for our later work, the cap $C_{p,t}$ expands to the full disk as $t\to 1$ and collapses toward the point $p$ as $t\to -1$. 

Define $R_p:\C\to\C$ to be reflection across the line through the origin that is perpendicular to $p$, so that
\begin{equation*}\label{eq:defRp}
R_p(z)=-p^2\bar{z} .
\end{equation*}
This reflection conjugates nicely under rotation, with
\begin{equation*} \label{eq:defRpconj}
R_p = p \circ R_1 \circ p^{-1} ,
\end{equation*}
and it conjugates the M\"{o}bius transformation according to
\begin{equation} \label{eq:Mobiusconj}
M_{pt} = R_p \circ M_{-pt} \circ R_p .
\end{equation}

Lastly, the hyperbolic reflection $\tau_C=\tau_{p,t}:\CD\rightarrow\CD$ associated with $C=C_{p,t}$ is defined by pulling back to the half-disk, reflecting, and then pushing out again:
\begin{equation} \label{eq:taudef}
\tau_{p,t}=M_{-pt}\circ R_p\circ M_{pt}.
\end{equation}
Clearly $\tau_C$ maps $C$ to $C^\star$, and vice versa, fixing the geodesic $\gamma$ inbetween. The hyperbolic reflection conjugates naturally under rotations, with 
\begin{equation*}\label{eq:rotationrelation}
\tau_{p,t} = p \circ \tau_{1,t} \circ p^{-1} ,
\end{equation*}
as one can check using the conjugation \eqref{eq:Mxiconj} for the M\"{o}bius map. Further, 
\begin{equation} \label{eq:tworot}
M_{-pt} = \tau_{p,t} \circ R_p \circ M_{pt} 
\end{equation}
by substituting \eqref{eq:Mobiusconj} into the right side of \eqref{eq:taudef}.

\section{\bf Conformal cap maps}\label{section:familyconfomap}
The next stage in constructing trial functions is to map each hyperbolic cap conformally to the whole disk. It is more convenient to map in the reverse direction, by describing maps $K_C = K_{p,t}:\D\rightarrow C_{p,t}$. Our goal is to  evaluate the limits of these maps for large and small caps, that is, as $t \to \pm 1$. 

\begin{proposition}[$0 \leq t < 1$]\label{lemma:disktobigcaps}
A family $K_{p,t}:\D \rightarrow C_{p,t}$ of conformal maps exists for $(p,t)\in S^1\times [0,1)$ such that as $t \to 1$  and $p \to q \in S^1$ one has 
    \[
    \text{$K_{p,t} \to \text{id.}$ \ locally uniformly on $\D$.}
    \]
\end{proposition}
\begin{proposition}[$-1 < t \leq 0$]\label{lemma:disktosmallcaps}
A family $K_{p,t}: \D \rightarrow C_{p,t}$ of conformal maps exists for $(p,t)\in S^1\times (-1,0]$ such that as $t\to-1$ and $p\to q \in S^1$, one has
  \[
  \text{$\tau_{p,t}\circ K_{p,t} \to R_q$ \ locally uniformly on $\D$.}
  \]
\end{proposition}
When $t=0$, the two propositions yield the same map $K_{p,0}$. Further the maps $K_{p,t}$ extend to $\overline{\D}$ and
\begin{equation} \label{eq:mapcont}
\text{$K_{p,t}(z)$ is continuous as a function of $(p,t,z) \in S^1 \times (-1,1) \times \overline{\D}$.}
\end{equation}
The proofs appear later in the section. 

\subsection*{Computations in the upper halfplane}\label{section:prooflemmacap}
Some of the needed calculations are more transparent in the halfplane. Define a M\"{o}bius transformation $W:\overline{\HH}\rightarrow\CD$ that wraps the halfplane onto the disk (\autoref{fig:Wmap}) by
\[
W(z) = \frac{i-z}{i+z} , \qquad z \in \overline{\HH} , 
\]
so that 
\[
W(0) = 1 , \quad W(\pm 1) = \pm i , \quad W(i) = 0.
\]
\begin{figure}[h]
\includegraphics[scale=0.5]{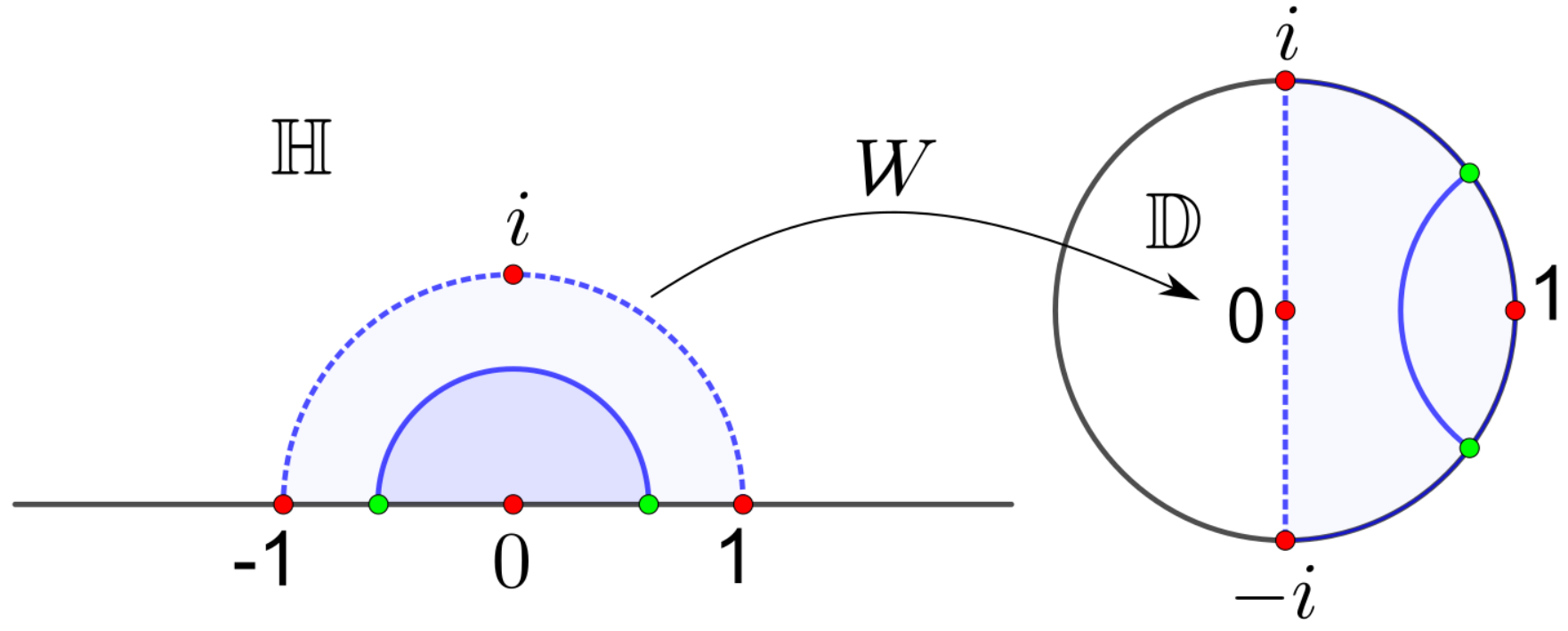}
\caption{\label{fig:Wmap} The map $W$ sends the origin to the point $1$. Half disks centered at the origin in the upper halfplane are mapped to hyperbolic caps in the disk.}
\end{figure}

The key fact is that $M_{1/3}$ on the disk corresponds to dilation by $1/2$ in the upper halfplane, since a direct calculation shows
\begin{equation}
(W^{-1} \circ M_{1/3} \circ W) (z) = \frac{1}{2} z . \label{eq:MWconj}
\end{equation}
%

Next, define a conformal map $S$ from the unit disk $\D$ to the doubly-slit plane
\[
\C \setminus \left( (-\infty,-1] \cup [1,\infty) \right) 
\]
by 
\[
S(z) = \frac{2}{z+1/z} = \frac{2z}{z^2+1} , \qquad z \in \D .
\]
The map satisfies
\[
S(\pm 1) = \pm 1 , \quad S(0) = 0 , \quad S^\prime(0)=2 .
\]
Clearly $S$ is symmetric in the horizontal axis, with $S(z)=\overline{S(\overline{z})}$, and $S$ maps the upper halfdisk $\D^+$ to the upper halfplane $\HH$. 

By rescaling $S$ and inverting, we define a map 
\begin{equation}
H_r(z) = r S^{-1} (2z/r) \label{eq:defr}
\end{equation}
from the halfplane $\HH$ to the halfdisk $\D^+(r)$ of radius $r>0$.  Note that
\[
H_r(0)= 0 , \quad H_r(\pm r/2) = \pm r .
\]
The factor of $2$ in the definition of $H_r$ ensures convergence to the identity, in the next lemma. 
\begin{lemma}\label{le:largecap}
	\[
	\text{$H_r \to \text{id.}$ \ locally uniformly on $\HH$, as $r \to \infty$.}
	\]
\end{lemma}
\begin{proof}
	Since $S^{-1}(0)=0$ and $(S^{-1})^\prime(0)=1/2$, the power series about the origin yields that 
	\begin{align*}
	H_r(z) = r S^{-1} \! \left( \frac{2z}{r} \right)
	& = r \left( \frac{1}{2} \frac{2z}{r} + O(2z/r)^2 \right) \\
	& = z + O(r^{-1}) \qquad \text{as $r \to \infty,$}
	\end{align*}
	where the error terms are uniform for $z$ belonging to a compact subset of the upper halfplane, since that ensures $|z|$ is bounded. 
\end{proof}

\subsection*{Convergence of the cap maps}

\begin{proof}[Proof of \autoref{lemma:disktobigcaps}]
For each cap $C=C_{p,t}$ with $t \in [0,1)$, let 
\[
K_C = K_{p,t} : \D \to C_{p,t}
\]
be the unique conformal map normalized by
\[
K_C(p) = p , \quad K_C \! \left( M_{p/3}(a) \right) = a , \quad K_C \! \left( M_{p/3}(b)\right) = b ,
\]
where $a$ and $b$ are the endpoints of the geodesic arc determining the cap. (This unusual  normalization of the endpoints is needed for proving convergence of $K_C$ to the identity map, as the cap expands to fill the whole disk. In effect, the endpoint normalization forces $K_C$ to ``push outward'' on the boundary near $p$, which counteracts the tendency of the map to ``pull inward'' as it compresses the disk into a cap.)

The maps satisfy a rotational conjugation that moves the center $p$ to the point $1$, namely 
\begin{equation}
K_{p,t} = p \circ K_{1,t} \circ p^{-1} , \label{eq:phiconjrelation}
\end{equation}
because each side of \eqref{eq:phiconjrelation} maps $\D$ conformally to the cap $C_{p,t}$, and the two sides agree at three points on the boundary, as follows. Each side maps $p$ to $p$. The right side maps $M_{p/3}(a)$ to $a$ as desired because
\begin{align*}
\left( p \circ K_{1,t} \circ p^{-1}\right) \left( M_{p/3}(a) \right) 
& = p \, K_{1,t} \left(M_{1/3} ( p^{-1}a) \right)  \qquad \text{by \eqref{eq:Mxiconj}} \\
& = p (p^{-1}a) = a 
\end{align*}
since $p^{-1}a$ is an endpoint for the cap $C_{1,t}$. Similarly each side of \eqref{eq:phiconjrelation} maps $M_{p/3}(b)$ to $b$. 

For the locally uniform convergence of $K_{p,t}$ to the identity as $t \to 1$, it suffices by the conjugation relation \eqref{eq:phiconjrelation} to prove the result for $p=1$, that is, to show
	\[
	\text{$K_{1,t} \to \text{id.}$ \ locally uniformly on $\D$, as $t \to 1$.}
	\]
	After conjugating with $W$ to transform the problem to the upper halfplane, the task further reduces to showing
	\begin{equation} 
	\text{$W^{-1} \circ K_{1,t} \circ W \to \text{id.}$ \ locally uniformly on $\HH$, as $t \to 1$.} \label{eq:taskreduced}
	\end{equation}
	Under the M\"{o}bius transformation $W^{-1}$, the cap $C_{1,t}$ in the disk transforms to a halfdisk of some radius $r$ centered at the origin in the halfplane, with $r$ depending in an increasing fashion on $t$. In particular, $r \to \infty$ as $t \to 1$ (expanding caps). 
	
	Recall now the conformal map $H_r$ defined in \eqref{eq:defr} that takes the halfplane $\HH$ to the halfdisk $\D^+(r)$, with $H_r(0)= 0$ and $H_r(\pm r/2) = \pm r$. We claim that 
	\begin{equation}
	W^{-1} \circ K_{1,t} \circ W = H_r . \label{eq:phipsiconj}
	\end{equation}
	Indeed, the left side maps $\HH$ conformally to $D^+(r)$, and maps $0$ to $0$. The left side also maps $r/2$ to $r$ (and $-r/2$ to $-r$), because 
	\begin{align*}
	& \left( W^{-1} \circ K_{1,t} \circ W \right) (r/2) \\
	& = (W^{-1} \circ K_{1,t}) \left( M_{1/3} \left( W(r) \right) \right) \qquad \text{by \eqref{eq:MWconj} with $z=r$} \\
	& = W^{-1} \left( W(r) \right) = r
	\end{align*}
	since $W(r)$ is an endpoint of the cap $C_{1,t}$. Hence the conformal maps on the two sides of \eqref{eq:phipsiconj} agree at three boundary points, and so must agree everywhere. 
	
	Thus we have reduced the task in \eqref{eq:taskreduced} to showing that $H_r \to \text{id.}$ locally uniformly on $\HH$ as $r \to \infty$, which is exactly the content of \autoref{le:largecap}. 
	
	The continuity of $K_{p,t}(z)$ as a function of $(p,t,z) \in S^1 \times [0,1) \times \overline{\D}$, which was asserted in \eqref{eq:mapcont}, follows from \eqref{eq:phiconjrelation} and \eqref{eq:phipsiconj} since $r$ depends continuously on $t$.
\end{proof}

\begin{proof}[Proof of \autoref{lemma:disktosmallcaps}]
For each cap $C=C_{p,t}$ with $t \in (-1,0]$, define the conformal map
\[
K_C = K_{p,t} : \D \to C_{p,t}
\]
in terms of the maps defined earlier with ``$t \in [0,1)$'' by letting
  \begin{gather}\label{eq:defphiptsmall}
   K_{p,t}=\tau_{p,t} \circ R_p \circ K_{p,-t} .
 \end{gather}
The image of the right side is indeed the cap $C_{p,t}$, because $K_{p,-t}$ maps onto $C_{p,-t}$, which reflects under $R_p$ to $C_{-p,-t}=C^\star_{p,t}$, which reflects hyperbolically under $\tau_{p,t}$ to $C_{p,t}$; or else more prosaically, compute that 
\[
\tau_{p,t} \circ R_p (C_{p,-t}) = C_{p,t}
\]
by using formula \eqref{eq:tworot} and the definition \eqref{eq:capdef} of the caps. 

When $t=0$, the two sides of \eqref{eq:defphiptsmall} are consistent since $\tau_{p,0}=R_p$ and $R_p \circ R_p$ is the identity. 

The definition \eqref{eq:defphiptsmall} implies that $\tau_{p,t}\circ K_{p,t} = R_p \circ K_{p,-t}$, which by \autoref{lemma:disktobigcaps} converges locally uniformly to $R_q \circ \text{id.} = R_q$ as $t \to -1$ and $p \to q$. That proves \autoref{lemma:disktosmallcaps}. 

Finally, continuity of $K_{p,t}(z)$ as a function of $(p,t,z) \in S^1 \times (-1,0] \times \overline{\D}$ follows from definition \eqref{eq:defphiptsmall} and the continuity proved earlier for the case ``$t \in [0,1)$''. 
\end{proof}

\section{\bf The Robin problem on the unit disk} \label{sec:diskproblem}
Our trial functions for the third eigenvalue on $\Omega$ will involve conformal transplantation of the second Robin eigenfunction of the unit disk, whose properties we now recall. 

In this section, the eigenfunctions satisfy
\[
\begin{split}
\Delta v + \lambda v & = 0 \quad \text{in $\D$,} \\
\partial_\nu v + \alpha v & = 0 \quad \text{on $\partial \D$.} 
\end{split}
\]
We do not rescale the Robin parameter here by the perimeter of the double disk. Thus the range 
$\alpha \in [-4\pi,0]$ in \autoref{thm:main} corresponds in this section to $\alpha \in [-1,0]$. Below we treat all $\alpha \in \R$, in the hope that \autoref{thm:main} might one day be extended to a larger range of $\alpha$-values.

The next two propositions and figures are taken from \cite[Section 5]{FreiLauW2018} and \cite[Section 5]{FreiLauS2018}. While the first Robin eigenfunction is not needed for our work, we present it anyway because it helps one's understanding to see the Robin groundstate in relation to the more familiar Neumann and Dirichlet cases. 
\begin{proposition}[First Robin eigenfunction of the disk]\label{basic1} The first Robin eigenvalue of the unit disk is simple, and changes sign at $\alpha=0$ according to 
\[
\lambda_1(\D;\alpha)
\begin{cases}
< 0 & \text{when $\alpha < 0$,} \\
= 0 & \text{when $\alpha = 0$,} \\
> 0 & \text{when $\alpha > 0$.}
\end{cases}
\]
The first eigenfunction is positive and radial, and is radially increasing when $\alpha< 0$, constant when $\alpha=0$, and radially decreasing when $\alpha > 0$.
\end{proposition}
\begin{figure}
\includegraphics[scale=0.6]{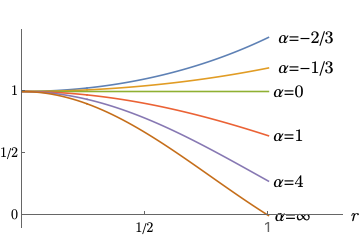}
\caption{\label{Robin_g_first}Plot of the (radially symmetric) first Robin eigenfunction of the unit disk, for various values of $\alpha$, normalized with height $1$ at the origin. When $\alpha=0$ it is the constant Neumann eigenfunction with eigenvalue $0$, and when $\alpha=\infty$ it is the Dirichlet eigenfunction $J_0(j_{0,1}r)$ with eigenvalue $j_{0,1}^2$. Between these extremes, the eigenfunction is $J_0(\sqrt{\lambda_1}\,r)$ where $\lambda_1=\lambda_1(\D;\alpha)>0$ is the eigenvalue. }
\end{figure}
\begin{figure}
\includegraphics[scale=0.6]{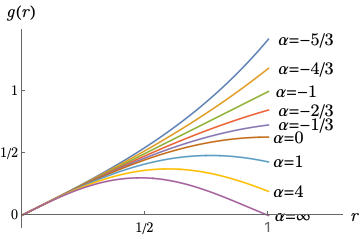}
\caption{\label{Robin_g_second}Plot of the radial part $g(r)$ of the second Robin eigenfunction of the unit disk, for various values of $\alpha$, normalized with $g^\prime(0)=1$. When $\alpha=-1$ one has $\lambda_2(\D;-1)=0$ and $g(r)=r$ is linear. When $\alpha>-1$ one has  $g(r)=(\text{const.})J_1(\sqrt{\lambda_2}\,r)$ where $\lambda_2=\lambda_2(\D;\alpha)>0$ is the eigenvalue. The eigenfunction is $g(r) e^{i \theta}$.}
\end{figure}
\begin{proposition}[Second Robin eigenfunctions of the disk]\label{basic2} The eigenfunction for $\lambda_2(\D;\alpha)$ can be taken in the form
\begin{equation} \label{eq:defv}
v = g(r) e^{i \theta} .
\end{equation}
The radial part has $g(0)=0,g^\prime(0)>0,g(r)>0$ for $r \in (0,1)$, and $g(1)>0$. When $\alpha \leq 0$ one finds $g(r)$ is strictly increasing, with $g^\prime(r)>0$. When $\alpha > 0$, the derivative $g^\prime$ is positive on some interval $(0,r_\alpha)$ and negative on $(r_\alpha,1)$, for some number $r_\alpha \in (0,1)$. 

The eigenvalue changes sign at $\alpha=-1$, with 
\[
\lambda_2(\D;\alpha)
\begin{cases}
< 0 & \text{when $\alpha < -1$,} \\
= 0 & \text{when $\alpha = -1$,} \\
> 0 & \text{when $\alpha > -1$.}
\end{cases}
\]
\end{proposition}
We will not need this fact, but the third eigenvalue $\lambda_3(\D;\alpha)$ of the disk equals the second eigenvalue, and has eigenfunction $g(r)e^{-i\theta}$.

The second eigenvalue of the disk can be evaluated explicitly when $\alpha>-1$ in terms of the Bessel function $J_1$, with $\lambda_2(\D;\alpha)=x(\alpha)^2$ where $x(\alpha) \in (0,j_{1,1})$ is the smallest positive solution of
\[
\frac{x J_1^\prime(x)}{J_1(x)} = - \alpha .
\]
This fact is derived in \cite[Section 5]{FreiLauW2018}, taking dimension $n=2$ there. 

\smallskip
The radial part $g$ of the second eigenfunction satisfies the following comparison result for mean values under conformal mapping, which will be central to proving \autoref{thm:main}. 
\begin{lemma}[Freitas and Laugesen \protect{\cite[Section 7]{FreiLauS2018}}] \label{le:areacompquoted}
Suppose $h :\Omega \to \D$ is a conformal map from a simply-connected planar domain $\Omega$ that has finite area. If $\alpha \leq 1$ then the radial part $g(r)$ of the eigenfunction for $\lambda_2(\D;\alpha)$ satisfies
%
\[
\frac{1}{\pi} \int_\D g^2 \, dA \leq \frac{1}{A(\Omega)} \int_\D g^2 |(h^{-1})^\prime|^2 \, dA = \frac{1}{A(\Omega)} \int_\Omega (g \circ h)^2 \, dA .
\]
%
Furthermore, if $\Omega$ is not a disk then the inequality is strict. 
\end{lemma}
Szeg\H{o}~\cite{Szego1954} proved this lemma under the assumption that $g$ is increasing, 
which holds for \autoref{thm:main} since $\alpha \leq 0$ there. 
Freitas and Laugesen~\cite[Section 7]{FreiLauS2018} extended Szeg\H{o}'s method to handle $\alpha \leq 1$, which in their paper was stated as $\alpha/2\pi \leq 1$ due to a different normalization. Also, their proof assumed $\Omega$ to have area $\pi$, but one may reduce to that case by rescaling $h$.


\section{\bf Hersch--Szeg\H{o} normalization, fold maps, and trial functions}
\label{sec:trialfuncts}
Trial functions on $\Omega$ are obtained in this section by precomposing the disk eigenfunction $v=g(r)e^{i\theta}$ with a ``folding map" and with the inverse of the cap map $K_{p,t}$, and with a M\"obius transformation performing the Hersch--Szeg\H{o} method of renormalization. This last step ensures that the trial function is orthogonal to the first eigenfunction $f_1$ on $\Omega$, for each $(p,t)\in S^1\times (-1,1)$. A topological argument will then be used to gain orthogonality also to the second eigenfunction $f_2$, for some particular choice of $(p,t)$.

\subsection*{Hersch--Szeg\H{o} normalization} 
Take $\mathcal{C}$ to be the space of hyperbolic caps on $\D$, parameterized by coordinates $(p,t)\in S^1\times (-1,1)$.
Building on the classical renormalization method of Szeg\H{o} \cite{Szego1954} and Hersch \cite{Hersch1970}, we prove the following:
\begin{lemma}[Orthogonality]\label{centerofmass}
Suppose $\Omega$ is a bounded planar domain and $f\in L^1(\Omega)$ is nonnegative, with $\int_{\Omega}f\,dA>0$. Let $g:[0,1]\rightarrow\R$ be a continuous function with  $0=g(0)<g(1)$. Let $v=g(r)e^{i\theta}$ on $\overline{\D}$. Suppose $T : \mathcal{C} \times \Omega \to \D$ is continuous, and write $T_C(z)=T(C,z)$. 
For each $C \in \mathcal{C}$, define a complex-valued function 
\[
V_C(w) = \int_\Omega (v \circ M_w \circ T_C) f \, dA , \qquad w \in \overline{\D} . 
\]
This $V_C$ is continuous, with $V_C(w) \neq 0$ when $w \in \partial \D$. Further, $V_C(w)=0$ for some $w \in\D$; and if in addition the function $g$ is strictly increasing then this vanishing point $w$ is unique and depends continuously on the cap $C$. 
\end{lemma}
The lemma and its proof are due essentially to Girouard, Nadirashvili and Polterovich \cite[Lemmas 2.2.4, 2.2.5 and 3.1.1]{GirNadPolt2009}. The assumption that $g(r)$ is strictly increasing holds true in their Neumann case $\alpha=0$ and also when $\alpha < 0$, but not when $\alpha>0$ (see \autoref{Robin_g_second}).

%
\begin{proof}[Proof of \autoref{centerofmass}]
Step 1 --- Existence. 
Notice $v=g(r)e^{i\theta}$ is continuous even at the origin, since $g(0)=0$. And $M_w(z)$ is continuous as a function of $(w,z) \in \overline{\D} \times \D$, taking values in $\overline{\D}$. In particular, $v \circ M_w \circ T_C$ is continuous and bounded on $\Omega$, and so $V_C(w)$ is well defined. Further, an application of dominated convergence shows that $V_C(w)$ is continuous as a function of $(C,w) \in \mathcal{C} \times \overline{\D}$.

The boundary behavior of $V_C$ is easily determined: when $w=e^{i\phi}$ one has $M_w(z)=e^{i\phi}$ for all $z \in \D$, and so 
\[
V_C(e^{i\phi}) = \left( g(1) \int_\Omega f \, dA \right) e^{i\phi} , \qquad \phi \in [0,2\pi] .
\]
Thus on the unit circle the continuous vector field $V_C$ is nonzero and points radially outward, remembering $g(1) \int_\Omega f \, dA > 0$ by assumption. Index theory now implies that $V_C$ vanishes somewhere in the interior of the disk. That is, $V_C(w)=0$ for some point $w = w(C) \in \D$. 

It remains to show this point $w(C)\in\D$ is unique and depends continuously on $C$. For these, we assume from now on that $g(r)$ is strictly increasing.

\smallskip
Step 2 --- Uniqueness. 
Fix $C \in \mathcal{C}$ and a corresponding point $w(C)$ constructed as above. Take $\psi \in \R$, and to simplify notation, let $\widetilde{T} = e^{i\psi} (M_{w(C)} \circ T_C) : \Omega \to \D$ so that $\widetilde{T}(z)$ is continuous. Define a new vector field
\[
\widetilde{V}(\xi) = \int_\Omega (v \circ M_\xi \circ \widetilde{T}) f \, dA , \qquad \xi \in \overline{\D} .
\]
Notice $\widetilde{V}(0)=e^{i\psi}V_C(w(C))=0$ by the choice of $w(C)$. Meanwhile because $f$ is nonnegative and $g$ is strictly increasing, $\widetilde{V}(\xi) \neq 0$ for all $\xi \neq 0$, by Girouard, Nadirashvili and Polterovich \cite[Lemma 3.1.1]{GirNadPolt2009} and the first paragraph in the proof of \cite[Lemma 2.2.4]{GirNadPolt2009}. 

To show $w(C)$ is the unique point at which $V_C$ can vanish, consider an arbitrary $w \in \D$ with $w \neq w(C)$ and decompose the M\"{o}bius map $M_w$ as 
\[
M_w(z) = M_\xi \! \left( e^{i\psi} M_{w(C)}(z) \right)
\]
where
\begin{gather*}
  e^{i\psi} = \frac{1-\overline{w(C)}w}{1-w(C)\overline{w}},\\
  \xi = M_w(-w(C)) \neq 0.
\end{gather*}
Then $V_C(w)=\widetilde{V}(\xi) \neq 0$ by above, and so $w(C)$ is the only point at which $V_C$ vanishes, proving uniqueness.

\smallskip
Step 3 --- Continuous dependence.
Suppose $C_n \to C$. Some subsequence $\{ w(C_{n_k}) \}$ converges to a point $w \in \overline{\D}$, and since $V_{C_n}(w(C_n))=0$ for each $n$ we conclude by letting $k \to \infty$ and using continuity that $V_C(w) = 0$. Hence $w=w(C)$ by uniqueness, and so $w(C_{n_k}) \to w(C)$ as $k \to \infty$. Applying this argument to each subsequence of the original $\{ C_n \}$ now shows that $w(C_n) \to w(C)$ as $n \to \infty$. That is, $w(C)$ depends continuously on $C$.
\end{proof}

Recall $f_1:\Omega\rightarrow\R$ is the eigenfunction corresponding to $\lambda_1(\Omega;\alpha/L)$. This groundstate does not change sign, and so we may choose $f_1>0$. 

From now on, fix $v=g(r)e^{i\theta}$ to be the eigenfunction corresponding to $\lambda_2(\D;\alpha/4\pi)$, as defined in \eqref{eq:defv}. The Robin parameter is $\alpha/4\pi$, with $\alpha \in \R$. 
Later when we prove \autoref{thm:main} we restrict to $\alpha \in [-4\pi,0]$.

Take a conformal map $T : \Omega \to \D$, so that the existence part of the Hersch--Szeg\H{o} \autoref{centerofmass} yields a $w \in \D$ with  
\begin{equation}\label{eq:NormalizedB}
  \int_\Omega (v \circ B) f_1 \,dA = 0 
\end{equation}
where $B = M_w \circ T$. For the rest of the paper we fix this normalized conformal map 
\[
B : \Omega \to \D .
\]

\subsection*{Fold map}
Given a hyperbolic cap $C=C_{p,t}$, define the ``fold map'' $F_C:\overline{\D}\to C$ by
\[
F_C(z) =
\begin{cases}
z & \text{if\ } z \in C, \\
\tau_C(z) & \text{if\ } z \in C^\star ,
\end{cases}
\]
where $\tau_C$ is hyperbolic reflection across the cap geodesic $\gamma$. 
We regard $F_C$ as folding the disk onto the cap $C$ across the geodesic $\gamma$.
This folding is two-to-one except on the geodesic, where it restricts to the identity. Clearly $F_C(z)$ depends continuously on $(C,z)$, in other words, on $(p,t,z)\in S^1\times (-1,1)\times \D$. 	

\subsection*{Trial functions and orthogonality}
Let $$G_C=(K_C)^{-1}:C\to\D$$ be the inverse of the conformal cap map defined in \autoref{section:familyconfomap}. For each hyperbolic cap $C$ and $w\in \overline{\D}$, define the trial function 
\[
u_{C,w}:\Omega\rightarrow\C
\]
by
\[
u_{C,w}=v\circ M_w\circ G_C \circ F_C \circ B,
\]
as shown schematically in \autoref{fig:Trial}. This function $u_{C,w}(z)$ is continuous as a function of $z \in \Omega$, is bounded by the maximum value of $|v|=g$, is smooth except along the preimage $B^{-1}(\gamma)$ of the geodesic defining $C$, and belongs to $H^1(\Omega)$ by conformal invariance of its Dirichlet energy (see the argument later for \eqref{eq:conformalinv}). 
\begin{figure}
\includegraphics[scale=0.1]{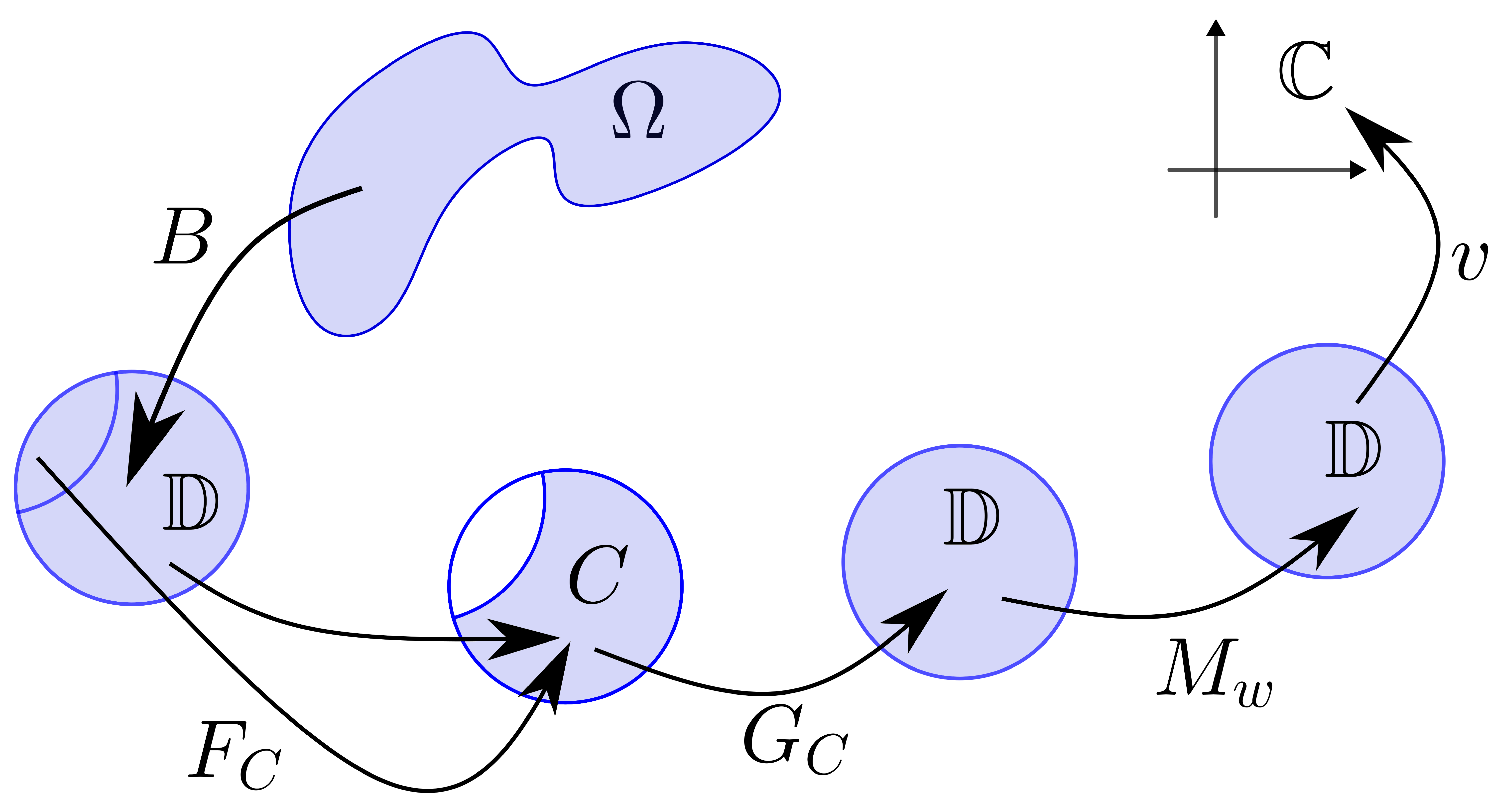}
\caption{\label{fig:Trial} Construction of the trial function $u_{C,w}$ on $\Omega$, by precomposing a disk eigenfunction $v$ with four maps. The M\"{o}bius parameter $w$ is chosen (for each cap $C$) to ensure orthogonality of the trial function to the first Robin eigenfunction on $\Omega$. The cap is then chosen by a topological argument to achieve orthogonality between the trial function and the second Robin eigenfunction.}
\end{figure}
\begin{lemma}[Continuous dependence of trial function]\label{lemma:continuousTrialFunction}
The function $u_{C,w}(z)$ depends continuously on $(C,w,z)$, in other words, on $(p,t,w,z)$. 
\end{lemma}
\begin{proof}
The M\"{o}bius map $M_w(z)$ is continuous as a function of $(w,z)$, and the conformal map $B(\cdot)$ is continuous. So by  definition of the trial function it suffices to show continuity of $(G_C \circ F_C)(z^*)$ as a function of $(C,z^*)$, where $z^* \in \D$. Notice the domain of $G_C$ is not fixed, being the cap $C$ itself, and so we must proceed carefully. 

Suppose $C_n \to C$ and $z_n^* \to z^* \in \D$. Write the parameters of $C_n$ as $(p_n,t_n)$ and those of $C$ as $(p,t)$, so that $p_n \to p$ and $t_n \to t$. The continuity goal is to show 
\begin{equation} \label{eq:inverselimit}
G_{p_n,t_n}(z_n)  \to G_{p,t}(z) \qquad \text{as $n \to \infty$,}
\end{equation}
where $z_n=F_{p_n,t_n}(z_n^*)$ and $z=F_{p,t}(z^*)$. Joint continuity of the fold map ensures $z_n \to z$, which will help in proving \eqref{eq:inverselimit}. 

Let $w_n=G_{p_n,t_n}(z_n)$. After passing to a subsequence we may suppose $w_n$ converges to some $w \in \overline{\D}$. Then $z = \lim_n z_n = \lim_n K_{p_n,t_n}(w_n) = K_{p,t}(w)$ by continuity in \eqref{eq:mapcont}. Hence $w=G_{p,t}(z)$, which proves the limit \eqref{eq:inverselimit} since the same argument applies to any subsequence of the original $\{ w_n \}$.
\end{proof}
Next we examine the limiting behavior of the trial functions as the caps expand to the full disk or shrink to a point. 
\begin{lemma}[Extension of trial function to large and small caps]\label{lemma:TrialFunctionExtension}
The function $u_{C,w}(z)$ with $C=C_{p,t}$ extends to $t=\pm 1$ as follows: 
\[
\begin{split}
    u_{C,w} \to v \circ M_{\widetilde{w}} \circ B & \qquad \text{as $p \to q \in S^1, t \to 1$ and $w \to \widetilde{w} \in \overline{\D}$,} \\
    u_{C,w} \to v \circ M_{\widetilde{w}} \circ R_q \circ B  & \qquad \text{as $p \to q \in S^1, t \to -1$ and $w \to \widetilde{w} \in \overline{\D}$,} 
\end{split}
\]
  where the convergence is locally uniform (and hence pointwise) on $\Omega$.
\end{lemma}
\begin{proof}
Part (i): $t \to 1$. Recall the definition $u_{C,w} = v \circ M_w \circ G_C \circ F_C \circ B$. If $E$ is a compact subset of $\D$ then $E \subset C$ whenever $t$ is sufficiently close to $1$, in which case the fold map $F_C$ equals the identity on $E$. Hence $F_C$ converges to the identity locally uniformly on $\D$. Next, $K_C : \D \to C$ converges locally uniformly to the identity as $t \to 1$, by \autoref{lemma:disktobigcaps}, and so its inverse map $G_C$ satisfies $G_C \to \text{id.}$ locally uniformly as $t \to 1$. Hence $u_{C,w} \to v \circ M_{\widetilde{w}} \circ B$ locally uniformly on $\Omega$ as $t \to 1$ and $w \to \widetilde{w}$.

\smallskip
Part (ii): $t \to -1$. If $E$ is a compact subset of $\D$ then $E \subset C^\star$ whenever $t$ is sufficiently close to $-1$, so that on $E$ the fold map $F_C$ is the hyperbolic reflection $\tau_C$. Observe that 
\[
G_C \circ \tau_C = (K_C)^{-1} \circ (\tau_C)^{-1} = (\tau_C \circ K_C)^{-1} 
\]
and $\tau_C \circ K_C \to R_q$ locally uniformly on $\D$ by \autoref{lemma:disktosmallcaps}. Hence 
\[
G_C \circ \tau_C \to (R_q)^{-1} = R_q
\]
locally uniformly, as $t \to -1$. Therefore $u_{C,w} = v \circ M_w \circ G_C \circ F_C \circ B \to v \circ M_{\widetilde{w}} \circ R_q \circ B$ locally uniformly on $\Omega$ as $t \to 1$.
\end{proof}

Joint continuity of the map $T_C(z) =(G_C \circ F_C \circ B)(z)$ was shown in the proof of \autoref{lemma:continuousTrialFunction}. Thus for each cap $C$, the Hersch--Szeg\H{o} \autoref{centerofmass} provides a point $w = w(C) \in \D$ for which the trial function is orthogonal to the first eigenfunction $f_1$, that is,
\begin{equation}\label{eq:herschunique}
\int_\Omega u_{C,w(C)} f_1 \,dA=0 .
\end{equation}
If $\alpha \leq 0$ (so that $g(r)$ is strictly increasing) then the point $w(C)=w(C_{p,t})$ is unique and depends continuously on the parameters $(p,t)$, by \autoref{centerofmass}.  

The next proposition shows that by choosing the cap correctly, the trial function can be made orthogonal also to the second eigenfunction $f_2$ on $\Omega$. This construction depends on $v \circ B$ being non-orthogonal to the second eigenfunction. If those functions are orthogonal, then we will use $v \circ B$ itself as a trial function when we later prove \autoref{thm:main}. 
\begin{proposition}[Orthogonality to 2nd eigenfunction]\label{propo:existenceCap}
If $\alpha \leq 0$ and $\int_\Omega (v \circ B) f_2 \,dA \neq 0$ then a hyperbolic cap $C$ exists for which
\begin{equation} \label{eq:herschunique2}
\int_{\Omega}u_{C,w(C)} f_2 \,dA=0 .
\end{equation}
\end{proposition}
\begin{proof}
Let $\zeta=\int_{\Omega}(v \circ B)f_2 \,dA \neq 0$. Define a complex valued function $\Phi$ on the cylinder $S^1\times(-1,1)$ by 
\[
\Phi(p,t)=\int_{\Omega}u_{p,t} \, f_2 \,dA , \qquad p \in S^1, \ t \in (-1,1) ,
\]
where 
\[
u_{p,t}=u_{C,w(C)}
\]
is the trial function associated with the cap $C=C_{p,t}$ and the normalizing point $w(C)$. Observe $\Phi(p,t)$ is continuous by dominated convergence, thanks to the continuous dependence of $u_{C,w}$ in \autoref{lemma:continuousTrialFunction} and continuity of $C \mapsto w(C)$. 

\autoref{lemma:asymptoticTrialFunction} below implies (again by dominated convergence) that $\Phi$ extends continuously to $t=1$, with 
\[
\Phi(q,1)=\int_{\Omega}(v\circ B) f_2 \, dA=\zeta , \qquad q \in S^1 .
\]
Thus $\Phi$ equals a nonzero constant at that end of the cylinder.

\autoref{lemma:asymptoticTrialFunction} also shows that $\Phi$ extends continuously to $t=-1$, with 
\[
\Phi(q,-1)=\int_{\Omega}(R_q \circ v \circ B) f_2 \, dA=R_q(\zeta) , \qquad q \in S^1.
\]
Thus the map $\Phi$ defines a homotopy between the loops
\[
\Phi(p,1)=\zeta \quad \text{and} \quad \Phi(p,-1)=R_p(\zeta) = -p^2 \overline{\zeta}
\]
in the complex plane. The first loop, being a nonzero constant, represents the trivial element of the fundamental group of the punctured plane $\C \setminus \{ 0 \}$. The second loop winds twice around the origin, and so represents a nontrivial element of that fundamental group. Therefore the loops cannot be homotopic in the punctured plane, and so the homotopy must pass through the origin at some point, meaning $\Phi(p,t)=0$ for some $(p,t)\in S^1\times(-1,1)$. The corresponding cap $C=C_{p,t}$ satisfies $\int_{\Omega}u_{C,w(C)}f_2\,dA=0$. 
\end{proof}

\begin{lemma}[Limit of $u_{p,t}$ for large and small caps]\label{lemma:asymptoticTrialFunction}
If $\alpha \leq 0$ then the trial functions $u_{p,t}$ converge locally uniformly (and hence pointwise) on $\Omega$, as follows:
  \[
  \begin{split}
    u_{p,t} \to v \circ B & \qquad\text{ as }\quad p \to q \in S^1 \text{ and }t \to 1,\\
    u_{p,t} \to R_q \circ v\circ B  &\qquad\text{ as }\quad p \to q \in S^1 \text{ and }t \to -1.
  \end{split}
  \]
\end{lemma}
\begin{proof}
The limiting behavior of $u_{C,w}$ as $t \to \pm 1$ and $w \to \widetilde{w}$ was determined in \autoref{lemma:TrialFunctionExtension}. Taking $w = w(C_{p,t})$, we see the task is to show $w(C_{p,t}) \to 0$ as $t \to \pm 1$, so that $\widetilde{w}=0$ and hence $M_{\widetilde{w}}=\text{id.}$\ in \autoref{lemma:TrialFunctionExtension}. That will immediately finish the proof when $t \to 1$, and when $t \to -1$ we need only observe also that $v$ and the reflection $R_q$ commute, 
\[
v \circ R_q = R_q \circ v ,
\]
because $v(R_q (re^{i\theta}))=R_q(v(re^{i\theta}))$ by a short computation using $R_q(z)=-q^2 \overline{z}$ and $v(re^{i\theta})=g(r)e^{i\theta}$. 

Part (i): $w(C_{p,t}) \to 0$ as $t \to 1$. Suppose $\{ w_k \}$ is a sequence of $w(C)$-values corresponding to some parameters $(p_k,t_k)$ with $t_k \to 1$. By passing to a subsequence we may suppose $w_k$ converges to some point $w^* \in \overline{\D}$. With the help of \autoref{lemma:TrialFunctionExtension}, boundedness of $v$ and dominated convergence, we may take the limit as $k \to \infty$ of the orthogonality condition \eqref{eq:herschunique} to find
\[
\int_\Omega (v \circ M_{w^*} \circ B) f_1 \,dA = 0.
\]
Since $\int_\Omega (v \circ M_0 \circ B) f_1 \,dA = 0$ by \eqref{eq:NormalizedB}, uniqueness in \autoref{centerofmass} (which holds since $g(r)$ is strictly increasing when $\alpha \leq 0$) implies that $w^*=0$. This holds for all sequences $\{ w_k \}$, and so $w(C) \to 0$ as $t \to 1$. 

\smallskip
Part (ii): $w(C_{p,t}) \to 0$ as $t \to -1$. Suppose $\{ w_k \}$ is a sequence of $w(C)$-values corresponding to some parameters $(p_k,t_k)$ with $t_k \to -1$ and $p_k \to q$. By passing to a subsequence we may suppose $w_k$ converges to some point $w^* \in \overline{\D}$. Again taking the limit as $k \to \infty$ of the orthogonality condition \eqref{eq:herschunique}, with the help of \autoref{lemma:TrialFunctionExtension} we find
\[
\int_\Omega (v \circ M_{w^*} \circ R_q \circ B) f_1 \, dA = 0.
\]
Also, the commutativity of $v$ and $R_q$ implies 
\[
\int_\Omega (v \circ M_0 \circ R_q \circ B) f_1 \,dA = R_q \int_\Omega (v \circ B) f_1 \,dA = 0
\]
by \eqref{eq:NormalizedB}. Uniqueness in \autoref{centerofmass} applied to the map $T=R_q \circ B$ now implies that $w^*=0$. This holds for all sequences $\{ w_k \}$, and so $w(C) \to 0$ as $t \to -1$. 
\end{proof}

\section{\bf An integral comparison}

The final ingredient needed for proving \autoref{thm:main} is an integral comparison on $\D$ and $\Omega$. Consider a trial function of the form
\[
u=u_{C,w}=v\circ (M_w\circ G_C \circ F_C \circ B) .
\]
Orthogonality is not required in this section, and so $C$ can be any cap and $w$ is any point in $\D$. 
\begin{lemma}\label{le:areacomp2}
If $\alpha \leq 4\pi$, then the radial part $g$ of the eigenfunction for $\lambda_2(\D;\alpha/4\pi)$ satisfies
\begin{equation*} \label{areabound}
\frac{1}{\pi}\int_\D |v|^2 \, dA < \frac{1}{A(\Omega)} \int_\Omega |u|^2 \, dA .
\end{equation*}
\end{lemma}
This result is similar to \autoref{le:areacompquoted}, except here $u$ and $v$ are related by a two-to-one map, whereas in the earlier lemma the map was one-to-one.
\begin{proof}
Define conformal maps 
\[
h :  B^{-1}(C) \to \D \qquad \text{and} \qquad k : B^{-1}(C^\star) \to \D 
\]
by letting $h = M_w \circ G_C \circ F_C \circ B$ on $B^{-1}(C)$ and letting $k = \overline{M_w} \circ G_C \circ F_C \circ B$ on $B^{-1}(C^\star)$ (see \autoref{fig:Trial}). The conformality of $h$ is clear, since the fold map is the identity on $C$. The fold is an anticonformal hyperbolic reflection on $C^\star$, but that effect is counteracted by the complex conjugate on $M_w$ in the definition of $k$, and so $k$ is conformal. 

Applying \autoref{le:areacompquoted} to $h$ and $k$ on their domains shows that 
\begin{align*}
\frac{A\left( B^{-1}(C) \right)}{\pi} \int_\D g^2 \, dA & \leq \int_\Omega (g \circ h)^2 \, dA , \\
\frac{A\left( B^{-1}(C^\star) \right)}{\pi} \int_\D g^2 \, dA & \leq \int_\D (g \circ k)^2 \, dA .
\end{align*}
At least one of these inequalities is strict by \autoref{le:areacompquoted}, since if $B^{-1}(C)$ and $B^{-1}(C^\star)$ were both disks then the domain $\Omega$, which is their union, would be disconnected. Adding the two inequalities now proves \autoref{le:areacomp2}.  
\end{proof}

\section{\bf Proof of inequality in \autoref{thm:main}}
\label{sec:main}

By scaling the domain $\Omega$ we may assume it has area $A=2\pi$. The goal is to prove
\[
\lambda_3(\Omega;\alpha/L) < \lambda_3(\D\sqcup\D;\alpha/4\pi) = \lambda_2(\D;\alpha/4\pi)
\]
when $\alpha \in [-4\pi,0]$.

\textbf{Case 1.} 
Suppose $\int_\Omega (v \circ B) f_2 \, dA = 0$. Recall $v$ is an eigenfunction for $\lambda_2(\D;\alpha/4\pi)$, and so it satisfies
\begin{equation} \label{eq:diskequation}
\lambda_2 \! \left(\D;\frac{\alpha}{4\pi}\right)\int_{\D} |v|^2 \, dA = \int_{\D}|\nabla v|^2 \, dA +\frac{\alpha}{2} g(1)^2 ,
\end{equation}
where we used in the boundary term that $|v|=g(1)$ on $\partial \D$. 

Take the trial function for $\Omega$ to be $u=v \circ B$, which is orthogonal to $f_1$ by \eqref{eq:NormalizedB} and orthogonal to $f_2$ by assumption in this Case. The variational characterization \eqref{eq:varchar} applied to $u$ gives that 
\begin{align*}
  \lambda_3 \! \left(\Omega;\frac{\alpha}{L}\right)\int_{\Omega} |u|^2 \,dA
&\leq
\int_\Omega |\nabla u|^2 \,dA+\frac{\alpha}{L}\int_{\partial\Omega} |u|^2 \,ds .
\end{align*}
Conformal invariance of the Dirichlet integral says that $\int_\Omega |\nabla u|^2 \,dA = \int_\D |\nabla v|^2 \, dA$, and $|u|=g(1)$ on $\partial \Omega$ since $B$ maps $\partial \Omega$ to $\partial \D$. So
\begin{align}
  \lambda_3 \! \left(\Omega;\frac{\alpha}{L}\right)\int_{\Omega} |u|^2 \,dA
& \leq \int_\D |\nabla v|^2 \,dA+\alpha g(1)^2 \notag \\
& < 2\int_\D |\nabla v|^2 \,dA+\alpha g(1)^2 \notag \\
& = 2 \lambda_2 \! \left(\D;\frac{\alpha}{4\pi}\right) \int_{\D} |v|^2 \, dA \label{ineq:beautiful2}
\end{align}
by \eqref{eq:diskequation}. Since $\Omega$ has area $2\pi$, \autoref{le:areacompquoted} implies that $2\int_{\D}|v|^2\,dA \leq \int_\Omega |u|^2 \, dA$.
This inequality can be substituted into the right side of \eqref{ineq:beautiful2} since $\lambda_2(\D;\alpha/4\pi) \geq 0$ (remember $\alpha/4\pi \geq -1$). Hence $\lambda_3(\Omega;\alpha/L) < \lambda_2(\D;\alpha/4\pi)$, which gives strict inequality in the theorem. 

\smallskip
\noindent 
\textbf{Case 2.} Suppose $\int_\Omega (v \circ B) f_2 \, dA \neq 0$. Then by \autoref{propo:existenceCap} (which requires $\alpha \leq 0$) a hyperbolic cap $C=C_{p,t}$ exists such that the trial function $u=u_{C,w(C)}:\Omega\rightarrow\C$ is orthogonal to the eigenfunctions $f_1$ and $f_2$ on $\Omega$, as in \eqref{eq:herschunique} and \eqref{eq:herschunique2}.
Hence by the variational characterization \eqref{eq:varchar}, 
\begin{align*}
  \lambda_3 \! \left(\Omega;\frac{\alpha}{L}\right)\int_{\Omega} |u|^2 \,dA
&\leq
\int_\Omega |\nabla u|^2 \,dA+\frac{\alpha}{L}\int_{\partial\Omega} |u|^2 \,ds .
\end{align*}
The Dirichlet integral on the right side splits into two parts, corresponding to the cap and the complementary cap, with 
\begin{align}
\int_\Omega |\nabla u|^2 \,dA
& =\int_{B^{-1}(C)} |\nabla (v\circ M_{w(C)}\circ G_{C} \circ B)|^2 \,dA \notag \\
& \quad +\int_{B^{-1}(C^\star)} |\nabla (v\circ M_{w(C)}\circ G_{C} \circ \tau_C \circ B)|^2 \,dA \notag \\
& = 2 \int_\D |\nabla v|^2 \, dA \label{eq:conformalinv}
\end{align}
by conformal invariance of the Dirichlet energy. Also, note that $|u(z)|=g(1)$ when $z\in \partial \Omega$, since the conformal maps take boundaries to boundaries. Hence
\begin{equation}\label{ineq:endgame}
  \lambda_3 \! \left(\Omega;\frac{\alpha}{L}\right)\int_{\Omega} |u|^2 \,dA
\leq
2\int_\D |\nabla v|^2 \,dA+\alpha g(1)^2.
\end{equation}
Applying identity \eqref{eq:diskequation} to the right side of \eqref{ineq:endgame}, we find
\begin{equation} \label{eq:beautiful}
\lambda_3 \! \left(\Omega;\frac{\alpha}{L}\right)\int_{\Omega} |u|^2 \,dA
\leq
2 \lambda_2 \! \left(\D;\frac{\alpha}{4\pi}\right) \int_{\D} |v|^2 \, dA.
\end{equation}
Since $\Omega$ has area $2\pi$, \autoref{le:areacomp2} implies that 
$$2\int_{\D}|v|^2\,dA < \int_\Omega |u|^2 \, dA.$$ 
Applying this inequality on the right side of \eqref{eq:beautiful} gives 
$$
\lambda_3 \! \left(\Omega;\frac{\alpha}{L}\right)\int_{\Omega} |u|^2 \, dA
<
\lambda_2 \! \left(\D;\frac{\alpha}{4\pi}\right) \int_{\Omega} |u|^2 \, dA$$
when $-4\pi < \alpha \leq 0$, where we used that $\lambda_2(\D;\alpha/4\pi)> 0$ for $\alpha > -4\pi$. Hence
$$\lambda_3 \! \left(\Omega;\frac{\alpha}{L}\right)
<
\lambda_2 \! \left(\D;\frac{\alpha}{4\pi}\right) .$$

It remains to handle $\alpha=-4\pi$. In that case $g(r)=r$ and $\lambda_2(\D;\alpha/4\pi)=0$, and so $\lambda_3(\Omega;-4\pi/L) \leq 0$ by \eqref{eq:beautiful}.
The inequality is strict, as follows. If equality held then equality would hold also in \eqref{ineq:endgame}, and so our trial function $u$ would actually be an eigenfunction for $\lambda_3(\Omega;-4\pi/L)$, and hence by elliptic regularity $u$ would be smooth on $\Omega$. Then $u \circ B^{-1} = v \circ (M_{w(C)} \circ G_C \circ F_C)$ would be smooth on $\D$, which in view of the fold map $F_C$ must mean that the normal derivative of $u \circ B^{-1}$ vanishes along the geodesic $\gamma$. Hence $v=re^{i\theta}$ has vanishing normal derivative along part of $\partial \D$, which is obviously false. This contradiction completes the proof that the inequality in the theorem is strict when $\alpha=-4\pi$.

\section{\bf Saturation in \autoref{thm:main}}
\label{sec:saturation}

To prove the inequality \eqref{ineq:main} in \autoref{thm:main} is asymptotically sharp, or saturates, we will show equality holds in the limit for the domains 
\[
\Omega_\e =  (\D - 1 + \e) \cup (\D + 1 - \e) 
\]
that approach the double disk
\[
\Omega_0 =  (\D - 1) \cup (\D + 1) 
\]
as $\e \to 0$.
\begin{figure}
  \includegraphics[scale=0.09]{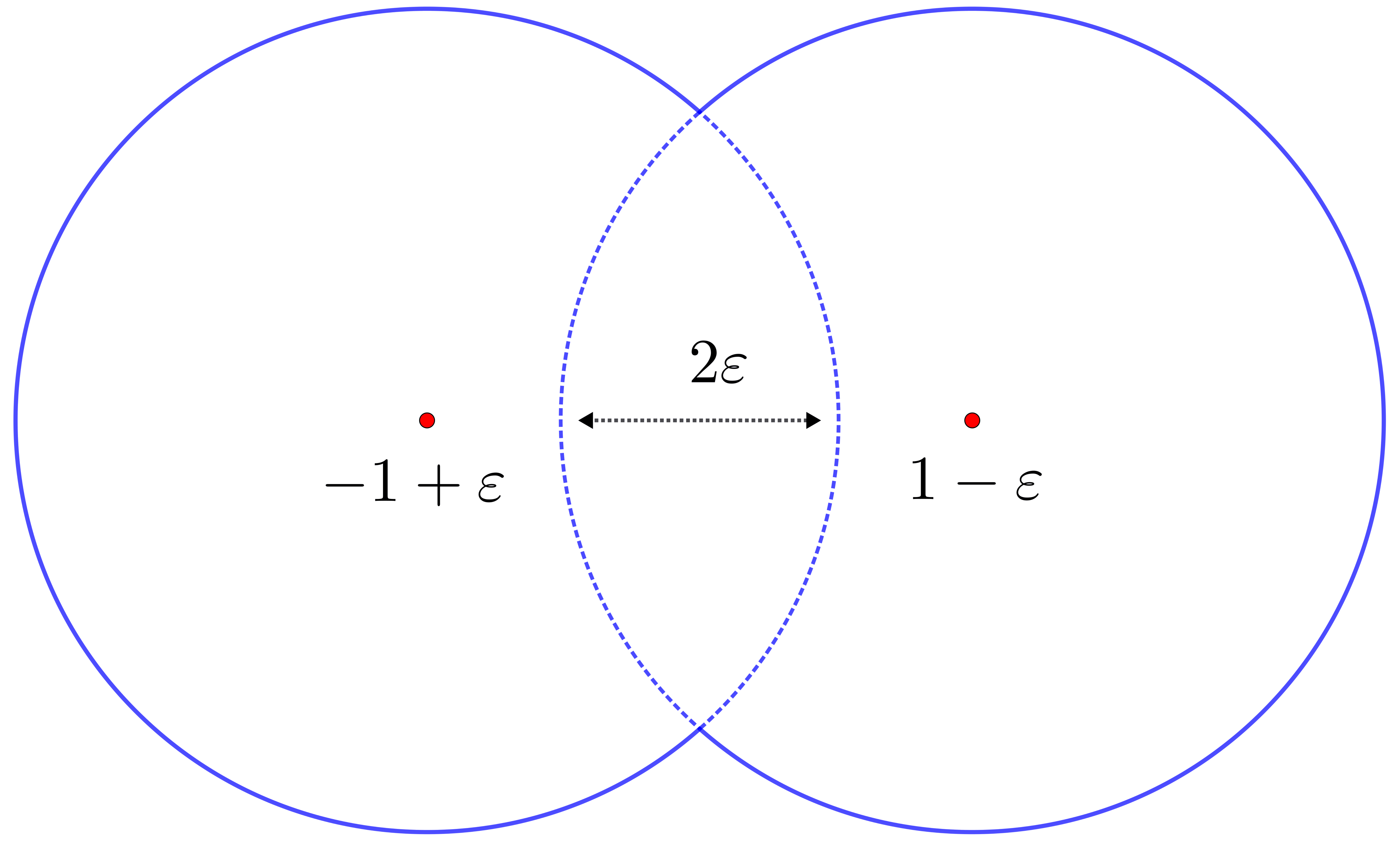}
  \caption{The domain $\Omega_\e=(\D - 1 + \e) \cup (\D + 1 - \e)$ that ``pulls apart'' to a union of two disks, as $\e \to 0$.}
\end{figure}
The double disk is not technically Lipschitz, due to its two-sided boundary point at the origin, but that obstruction to defining the spectrum can be avoided just by moving the disks farther apart. Let $A(\e)$ and $L(\e)$ be the area of $\Omega_\e$ and the length of its boundary, respectively, so that 
\[
\lambda_3(\Omega_\e;\alpha/L(\e)) A(\e) < \lambda_3(\Omega_0;\alpha/4\pi)2\pi
\]
by \autoref{thm:main}. 

Since $A(\e) \to 2\pi$ as $\e \to 0$, saturation in \autoref{thm:main} will follow from proving 
\begin{equation}\label{ineq:mainsharp}
 \liminf_{\e \to 0} \lambda_j(\Omega_\e;\alpha/L(\e)) \geq \lambda_j(\Omega_0;\alpha/4\pi) 
\end{equation}
for each $j \geq 1$ and $\alpha \in \R$. (We need the case $j=3$ and $\alpha \in [-4\pi,0]$.)

The idea is to compare the Robin spectrum on $\Omega_\e$ with the Robin spectrum on a weighted double disk, by ``pulling apart'' the values of the trial function and multiplying by weight $1/2$ in the ``overlap''  region. So construct a linear transformation 
\[
T_\e : H^1(\Omega_\e) \to H^1(\Omega_0)
\]
by 
\[
T_\e u(z) = 
\begin{cases}
u(z+\e) & \text{when $z \in \D-1$,} \\
u(z-\e) & \text{when $z \in \D+1$,} 
\end{cases}
\]
where $u \in H^1(\Omega_\e)$ is arbitrary. 
Let 
\[
\Theta_\e = \left[ (\D - 1) \cap (\D + 1 - 2\e) \right] \cup \left[ (\D + 1) \cap (\D - 1 + 2\e) \right] ,
\]
so that $\Theta_\e \subset \Omega_0$ and $\Theta_\e$ shrinks toward the origin as $\e \to 0$.  Define interior and boundary weights on the double disk by 
\[
\rho_\e = 
\begin{cases}
1 & \text{in $\Omega_0 \setminus \Theta_\e$} \\
1/2 & \text{in $\Omega_0 \cap \Theta_\e$}
\end{cases} ,
\qquad 
\beta_\e = 
\begin{cases}
1 & \text{in $\partial \Omega_0 \setminus \partial \Theta_\e$} \\
0 & \text{in $\partial \Omega_0 \cap \partial \Theta_\e$}
\end{cases} .
\]
\begin{figure}
  \includegraphics[scale=0.09]{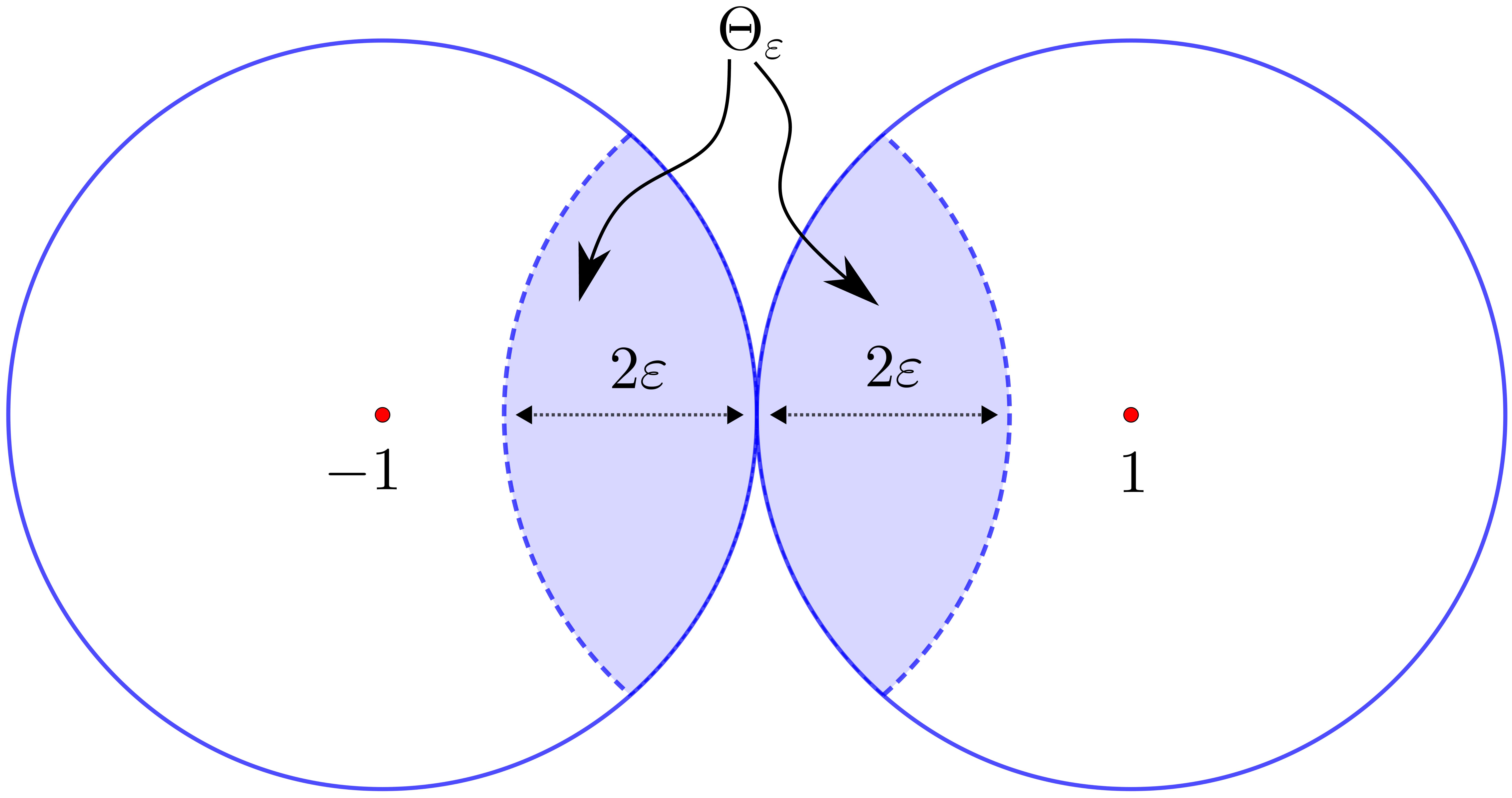}
\caption{\label{fig:Omega0} The double disk $\Omega_0$, and its subdomain $\Theta_\e$.}
\end{figure}

In terms of these weights, the Rayleigh quotient of $u \in H^1(\Omega_\e)$ pulls apart to 
\begin{align}
Q_{\alpha/L(\e)}(u) 
& = \frac{\int_{\Omega_\e} |\nabla u|^2 \, dA + (\alpha/L(\e)) \int_{\partial\Omega_\e} |u|^2 \, ds}{\int_{\Omega_\e}|u|^2 \,dA} \notag \\
& = \frac{\int_{\Omega_0} |\nabla (T_\e u)|^2 \, \rho_\e \, dA + (\alpha/L(\e)) \int_{\partial\Omega_0} |T_\e u|^2 \, \beta_\e \, ds}{\int_{\Omega_0}|T_\e u|^2 \, \rho_\e \,dA} \notag \\
& \equiv Q(v;\rho_\e,\alpha \beta_\e/L(\e)) \label{eq:Qrelation}
\end{align}
where $v=T_\e u \in H^1(\Omega_0)$. 

Therefore the minimax characterization of the $j$th eigenvalue, with $U$ ranging over $j$-dimensional subspaces of $H^1(\Omega_\e)$ and $V$ ranging over $j$-dimensional subspaces of $H^1(\Omega_0)$, implies that
\begin{align}
\lambda_j(\Omega_\e;\alpha/L(\e)) 
& = \min_U \max_{u \in U} Q_{\alpha/L(\e)}(u) \notag \\
& = \min_U \max_{v \in T_\e(U)} Q(v; \rho_\e,\alpha \beta_\e/L(\e)) \quad \text{by \eqref{eq:Qrelation}} \notag \\
& \geq \min_V \max_{v \in V} Q(v; \rho_\e,\alpha \beta_\e/L(\e)) \notag \\
& = \lambda_j(\Omega_0;\rho_\e,\alpha \beta_\e/L(\e)) , \label{eq:lambda_ineq}
\end{align}
where the inequality relies on $T_\e(U) = \{ T_\e u : u \in U\}$ being a $j$-dimensional subspace of $H^1(\Omega_0)$, which is easily checked. 

Since $L(\e) \to 4\pi$ and $\rho_\e \to 1$ pointwise on $\Omega_0$ and $\beta_\e \to 1$ pointwise on $\partial \Omega_0$, with $1/2 \leq \rho_\e \leq 1$ and $|\beta_\e| \leq 1$ for all $\e$, we conclude from  \autoref{pr:spectralconvergence} in the Appendix that
\begin{equation*} \label{eq:formal}
\lim_{\e \to 0} \lambda_j(\Omega_0;\rho_\e,\alpha \beta_\e/L(\e)) = \lambda_j(\Omega_0;1,\alpha/4\pi) .
\end{equation*}
This limit, together with \eqref{eq:lambda_ineq}, proves inequality \eqref{ineq:mainsharp}.

\section{\bf Proof of \autoref{cor:steklov}}
\label{sec:steklov}

Notice $\sigma$ belongs to the Steklov spectrum exactly when $0$ belongs to the Robin spectrum for parameter value $\alpha=-\sigma$. Further, since the Robin eigenvalues are increasing with respect to $\alpha$ and the Steklov spectrum is discrete, each Robin eigenvalue $\lambda_j(\Omega;\alpha)$ can equal $0$ for at most one value of $\alpha$. 

One has that 
\begin{align*}
\lambda_2 (\Omega;-4\pi/L) A 
& \leq \lambda_3(\Omega;-4\pi/L) A \\
& < \lambda_3(\D\sqcup\D;-1) 2\pi \quad \text{by  \autoref{thm:main} with $\alpha=-4\pi$} \\
& = \lambda_2(\D;-1) 2\pi \\
& = 0 
\end{align*}
by a formula in \autoref{basic2}. Also, 
\[
\lambda_3(\Omega;0) \geq \lambda_2(\Omega;0) > \lambda_1(\Omega;0) = 0 .
\]
Since the Robin eigenvalues vary continuously with $\alpha$, the preceding observations imply there must exist values $-4\pi < \alpha_3 \leq \alpha_2 < 0$ for which $\lambda_3(\Omega;\alpha_3/L) = 0$ and $\lambda_2(\Omega;\alpha_2/L) = 0$. It follows that $-\alpha_3/L$ is the second positive Steklov eigenvalue $\sigma_2(\Omega)$, and so  
\[
\sigma_2(\Omega) L(\Omega) = -\alpha_3 < 4\pi , 
\] 
which is the desired inequality. 

The asymptotic equality statement from \autoref{thm:main} implies that if $-4\pi < \alpha \leq 0$ then the domain $\Omega_\e = (\D - 1 + \e)\cup (\D + 1 - \e)$ satisfies
\begin{align*}
\lim_{\e \to 0} \lambda_3\left( \Omega_\e;\alpha/L(\Omega_\e) \right) A(\Omega_\e) 
& = \lambda_2(\D;\alpha/4\pi) 2\pi \\
& > \lambda_2(\D;-1) 2\pi \\
& = 0 .
\end{align*}
Hence whenever $\e$ is sufficiently close to $0$, the number $\alpha_3=\alpha_3(\e)$ corresponding to the domain $\Omega_\e$ satisfies $-4\pi<\alpha_3(\e) < \alpha$. Since $\alpha$ can be chosen arbitrarily close to $-4\pi$, we conclude $\alpha_3(\e) \to -4\pi$ as $\e \to 0$. That is, $\sigma_2(\Omega_\e) L(\Omega_\e) \to 4\pi$ as $\e \to 0$.

\appendix
\section{\bf Robin spectrum --- existence and convergence}
\label{sec:appendix}

Existence and convergence results on the Robin spectrum, with weight functions bounded above and below, are needed for proving asymptotic sharpness (saturation) of \autoref{thm:main}. We state these results in $n$ dimensions, since they apply equally well there. Write $\partial_\nu u$ for the outward normal derivative of $u$ at the boundary.
\begin{proposition}[Existence of Robin eigenvalues]\label{pr:spectralexistence}
Suppose $\Omega \subset \Rn$ is a bounded Lipschitz open set. If the weight functions $\rho : \Omega \to \R, \omega : \Omega \to \R$ and $\beta : \partial \Omega \to \R$ are measurable with
\[
a^{-1} \leq \rho \leq a , \qquad a^{-1} \leq \omega \leq a , \qquad |\beta| \leq a ,
\]
for some $a>1$, then there exist functions \text{$u_1,u_2,u_3,\ldots \in
H^1(\Omega)$} and numbers
\[
\lambda_1 \leq \lambda_2 \leq \lambda_3 \leq \cdots \to \infty
\]
such that $\{ u_j \}$ is an orthonormal basis for $L^2(\omega \, dx)$ and $u_j$ is a Robin eigenfunction with eigenvalue $\lambda_j$, meaning
\[
\begin{split}
\nabla \cdot (\rho \nabla u) + \lambda \omega u & = 0 \quad \text{in $\Omega$,} \\
\rho \partial_\nu u + \beta u & = 0 \quad \text{on $\partial \Omega$,} 
\end{split}
\]
in the weak sense with respect to test functions in $H^1(\Omega)$. 

Further, the decomposition $f = \sum_j \langle f , u_j \rangle_{L^2(\omega \, dx)} \, u_j$ holds with convergence in $L^2(\Omega)$ for each $f \in L^2(\Omega)$, and holds with convergence in $H^1(\Omega)$ for each $f \in H^1(\Omega)$.
\end{proposition}
\begin{proof}
The Rayleigh quotient for this problem is
\[
u \mapsto \frac{\int_{\Omega}|\nabla u|^2 \, \rho \,dx + \int_{\partial\Omega}|u|^2 \, \beta \, dS}{\int_{\Omega} |u|^2 \, \omega \, dx}.
\]
The denominator of the Rayleigh quotient is comparable to $\int_\Omega |u|^2 \, dx$, due to the upper and lower bounds on the weight $\omega$.

The numerator is coercive on $H^1(\Omega)$ after adding a large multiple of $\int_\Omega |u|^2 \, \omega \, dx$, since 
 \begin{align*}
 & \int_\Omega |\nabla u|^2 \rho \, dx + \int_{\partial \Omega} |u|^2 \beta \, dS \\
& \geq a^{-1} \int_\Omega |\nabla u|^2 \, dx - a \int_{\partial \Omega} |u|^2 \, dS \qquad \text{by the bounds on $\rho$ and $\beta$} \\
   & \geq a^{-1} \int_\Omega |\nabla u|^2 \, dx - C_1 \left( \int_\Omega |\nabla u| |u| \, dS + \int_\Omega |u|^2 \, dx \right) \qquad \text{by \eqref{eq:traceconsequence} below} \\
& \geq (a^{-1}/2) \lVert u \rVert_{H^1(\Omega)}^2 - C_2 \int_\Omega |u|^2 \, \omega \, dx
 \end{align*}
 by Cauchy-with-$\epsilon$.
Here the constants $C_1,C_2>0$ depend on $a$, and also on $\Omega$ through the bound 
\begin{equation}
\int_{\partial \Omega} |u|^2 \, dS \leq C \left( \int_\Omega \lvert \nabla u \rvert \lvert u \rvert \, dx + \int_\Omega |u|^2 \, dx \right) \label{eq:traceconsequence}
\end{equation}
that can be found in the proof of the trace theorem \cite[{\S}4.3]{EvansGariepy}.
 
 The proposition now follows from the discrete spectral theorem for sesquilinear forms \cite[Section 6.3]{BlanchardBruning}.
\end{proof}

If the weight functions converge pointwise then the spectrum should converge too. That is the content of the next proposition. In order to emphasize the dependence on the weights, we write $\lambda_j(\rho,\omega,\beta)$ for the eigenvalues constructed in \autoref{pr:spectralexistence}. 
\begin{proposition}[Convergence of Robin eigenvalues]\label{pr:spectralconvergence}
Suppose $\Omega \subset \Rn$ is a bounded Lipschitz domain, and $\rho_\e : \Omega \to \R, \omega_\e : \Omega \to \R$  and $\beta_\e : \partial \Omega \to \R$ are measurable for each $\e \geq 0$ with
\[
a^{-1} \leq \rho_\e \leq a , \qquad a^{-1} \leq \omega_\e \leq a , \qquad |\beta_\e| \leq a ,
\]
for some $a>1$. If $\rho_\e \to \rho_0, \omega_\e \to \omega_0$ and $\beta_\e \to \beta_0$ pointwise as $\e \to 0$ then 
\[
\lambda_j(\rho_\e,\omega_\e,\beta_\e) \to \lambda_j(\rho_0,\omega_0,\beta_0)
\]
as $\e \to 0$, for each $j \geq 1$.
\end{proposition}
\begin{proof}
Fix $j \geq 1$, and let $\lambda_j(\e) = \lambda_j(\rho_\e,\omega_\e,\beta_\e)$. The minimax variational characterization \cite[Section 6.1]{BlanchardBruning} says  
\begin{equation} \label{eq:varchargeneral}
\lambda_j(\e) = \min_U \max_{u \in U} Q_\e(u)
\end{equation}
where $U$ ranges over $j$-dimensional subspaces of $H^1(\Omega)$ and 
\[
Q_\e(u) = \frac{\int_\Omega |\nabla u|^2 \rho_\e \, dx + \int_{\partial \Omega} |u|^2 \beta_\e \, dS}{\int_\Omega |u|^2 \omega_\e \, dx} , \qquad u \in H^1(\Omega) ,
\]
is the Rayleigh quotient.  

The maximum in the variational characterization \eqref{eq:varchargeneral} is really taken over a $(j-1)$-dimensional sphere of coefficients, because if $v_1,\dots,v_j$ is a basis for $U$ then each nonzero $u \in U$ can be written $u=t(b_1 v_1 + \dots + b_j v_j )$ with $t>0$ and $|b_1|^2 + \dots + |b_j|^2 = 1$, so that $Q_\e(u) = Q_\e(b_1 v_1 + \dots + b_j v_j)$ by homogeneity of the Rayleigh quotient. 

The minimum in \eqref{eq:varchargeneral} is attained when $U$ equals the subspace $U_\e$ spanned by the first $j$ eigenfunctions $u_{1,\e},\dots,u_{j,\e}$ corresponding to weights $\rho_\e, \omega_\e, \beta_\e$. That is, $\lambda_j(\e) = \max_{u \in U_\e} Q_\e(u)$. 

Choosing $U=U_0$ in \eqref{eq:varchargeneral} gives that $\max_{u \in U_0} Q_\e(u) \geq \lambda_j(\e)$. Hence  
\begin{equation}
\lambda_j(0) = \max_{u \in U_0} Q_0(u) \geq \limsup_{\e \to 0} \max_{u \in U_0} Q_\e(u) \geq \limsup_{\e \to 0} \lambda_j(\e) , \label{eq:limsupbound}
\end{equation}
where the first inequality holds by dominated convergence, using the boundedness of $\rho_\e, \omega_\e,\beta_\e$ together with their pointwise convergence to $\rho_0, \omega_0, \beta_0$.   

For an inequality in the reverse direction, take an arbitrary sequence of $\e$-values approaching $0$. The eigenvalues are bounded by \eqref{eq:limsupbound} and so the $L^2$-normalized eigenfunctions $u_{1,\e},\dots,u_{j,\e}$ are bounded in $H^1(\Omega)$ by coercivity of the Rayleigh quotient. Thus we may pass to a subsequence such that each eigenfunction $u_{k,\e}$ converges weakly in $H^1(\Omega)$ to some limiting function $u_k$. Then $u_{k,\e} \to u_k$ in $L^2(\Omega)$ by the compact imbedding of $H^1$ into $L^2$, and $u_{k,\e} \to u_k$ in $L^2(\partial \Omega)$ by using the trace bound \eqref{eq:traceconsequence}. 

The functions $u_k$ are orthonormal in $L^2(\omega_0 \, dx)$, since the eigenfunctions $u_{k,\e}$ are orthonormal in $L^2(\omega_\e \, dx)$:
\begin{equation}
\int_\Omega u_k u_\ell \, \omega_0 \, dx = \lim_{\e \to 0} \int_\Omega u_k u_\ell \, \omega_\e \, dx = \lim_{\e \to 0} \int_\Omega u_{k,\e} u_{\ell,\e} \, \omega_\e \, dx = \delta_{k,\ell} \label{eq:orthoeps}
\end{equation}
by dominated convergence and $L^2(\Omega)$-convergence, using the uniform bound on the $\omega_\e$ weights. Thus the subspace $U$ spanned by $\{ u_1,\dots,u_j \}$ is $j$-dimensional. Given an arbitrary $u = c_1 u_1 + \dots + c_j u_j \in U$ we let 
\[
u_\e = c_1 u_{1,\e} + \dots + c_j u_{j,\e} \in U_\e
\]
and deduce that $u_\e \rightharpoonup u$ weakly in $H^1(\Omega)$, $u_\e \to u$ in $L^2(\Omega)$ and $u_\e \to u$ in $L^2(\partial \Omega)$. Therefore
\begin{align*}
\int_\Omega |\nabla u|^2 \, \rho_0 \, dx 
& = \lim_{\e \to 0} \int_\Omega \nabla u \cdot \nabla u_\e \, \rho_0 \, dx \qquad \text{by weak convergence in $H^1$} \\
& = \lim_{\e \to 0} \int_\Omega (\rho_0^{1/2} \nabla u) \cdot (\rho_\e^{1/2}  \nabla u_\e) \, dx \\
& \hspace*{-2cm} \text{using that $\rho_\e^{1/2} \rho_0^{1/2} \nabla u \to \rho_0 \nabla u$ in $L^2(\Omega)$ by dominated convergence} \\
& \leq \left( \int_\Omega |\nabla u|^2 \, \rho_0 \, dx \right)^{\! 1/2} \liminf_{\e \to 0} \left( \int_\Omega |\nabla u_\e|^2 \, \rho_\e \, dx \right)^{\! 1/2} ,
\end{align*}
from which we conclude $\int_\Omega |\nabla u|^2 \, \rho_0 \, dx \leq \liminf_{\e \to 0} \int_\Omega |\nabla u_\e|^2 \, \rho_\e \, dx$.  Further, $\int_\Omega |u|^2 \, \omega_0 \, dx = \lim_{\e \to 0} \int_\Omega |u_\e|^2 \, \omega_\e \, dx$ by dominated convergence and $L^2$-convergence (like in \eqref{eq:orthoeps}), and $\int_{\partial \Omega} |u|^2 \, \beta_0 \, dS = \lim_{\e \to 0} \int_{\partial \Omega} |u_\e|^2 \, \beta_\e \, dS$ by reasoning similarly.  

Combining the last three observations, we find
\begin{equation}
Q_0(u) \leq \liminf_{\e \to 0} Q_\e(u_\e) . \label{eq:Q0liminf}
\end{equation}
Using $U$ in the variational characterization with $\e=0$ and then applying \eqref{eq:Q0liminf} gives that
\begin{equation}
\lambda_j(0)
\leq
\max_{u \in U} Q_0(u)
\leq 
\liminf_{\e \to 0} \max_{u \in U_\e} Q_\e(u)
=
\liminf_{\e \to 0} \lambda_j(\e) , \label{eq:liminfbound}
\end{equation}
where the $\liminf$ runs through only the subsequence of $\e$-values constructed above. The original sequence of $\e$-values was arbitrary, though, and so the last formula holds for all $\e \to 0$. 

Combining the $\limsup$ and $\liminf$ bounds in \eqref{eq:limsupbound} and \eqref{eq:liminfbound} proves the convergence of the $j$th eigenvalue, as wanted for the proposition. 
\end{proof}

\section*{Acknowledgments}
This research was supported by a grant from the Simons Foundation (\#429422 to Richard Laugesen). Alexandre Girouard acknowledges support from the Natural Sciences and Engineering Research Council of Canada (NSERC).

\nocite{*}

\bibliographystyle{plain}
\bibliography{biblio}

\begin{thebibliography}{10}

\bibitem{AntunesFreitas2012}
P.~R.~S. Antunes and P.~Freitas.
\newblock Numerical optimization of low eigenvalues of the {D}irichlet and
  {N}eumann {L}aplacians.
\newblock {\em J. Optim. Theory Appl.}, 154(1):235--257, 2012.

\bibitem{AntunesFreitasKrejcirik2017}
P.~R.~S. Antunes, P.~Freitas, and D.~Krej\v{c}i\v{r}\'{\i}k.
\newblock Bounds and extremal domains for {R}obin eigenvalues with negative
  boundary parameter.
\newblock {\em Adv. Calc. Var.}, 10(4):357--379, 2017.

\bibitem{BlanchardBruning}
P.~Blanchard and E.~Br\"{u}ning.
\newblock {\em Variational Methods in Mathematical Physics. A Unified
  Approach}.
\newblock Texts and Monographs in Physics. Springer--Verlag, Berlin, 1992.
\newblock Translated from the German by Gillian M. Hayes.

\bibitem{BucFreiKen2017}
D.~Bucur, P.~Freitas, and J.~Kennedy.
\newblock The {R}obin problem.
\newblock In {\em Shape Optimization and Spectral Theory}, pages 78--119. De
  Gruyter Open, Warsaw, 2017.

\bibitem{BucurHenrot2019}
D.~Bucur and A.~Henrot.
\newblock Maximization of the second non-trivial {N}eumann eigenvalue.
\newblock {\em Acta Math.}, 222(2):337--361, 2019.

\bibitem{CianciKarpukhinMedvedev}
D.~Cianci, M.~Karpukhin, and V.~Medvedev.
\newblock On branched minimal immersions of surfaces by first eigenfunctions.
\newblock {\em Ann. Global Anal. Geom.}, to appear.
\newblock arXiv:1711.05916.

\bibitem{ElSoufiGiacominiJazar}
A.~El~Soufi, H.~Giacomini, and M.~Jazar.
\newblock A unique extremal metric for the least eigenvalue of the {L}aplacian
  on the {K}lein bottle.
\newblock {\em Duke Math. J.}, 135(1):181--202, 2006.

\bibitem{ElSoufiIlias}
A.~El~Soufi and S.~Ilias.
\newblock Le volume conforme et ses applications d'apr\`es {L}i et {Y}au.
\newblock In {\em S\'{e}minaire de {T}h\'{e}orie {S}pectrale et
  {G}\'{e}om\'{e}trie, {A}nn\'{e}e 1983--1984}, pages VII.1--VII.15. Univ.
  Grenoble I, Saint-Martin-d'H\`eres, 1984.

\bibitem{EvansGariepy}
L.~C. Evans and R.~F. Gariepy.
\newblock {\em Measure Theory and Fine Properties of Functions}.
\newblock Studies in Advanced Mathematics. CRC Press, Boca Raton, FL, 1992.

\bibitem{FreiLauW2018}
P.~Freitas and R.~S. Laugesen.
\newblock From {N}eumann to {S}teklov and beyond, via {R}obin: the {W}einberger
  way.
\newblock arXiv:1810.07461.

\bibitem{FreiLauS2018}
P.~Freitas and R.~S. Laugesen.
\newblock From {S}teklov to {N}eumann and beyond, via {R}obin: the {S}zeg{\H o}
  way.
\newblock {\em Canad. J. Math.}, to appear.

\bibitem{GirNadPolt2009}
A.~Girouard, N.~Nadirashvili, and I.~Polterovich.
\newblock Maximization of the second positive {N}eumann eigenvalue for planar
  domains.
\newblock {\em J. Differential Geom.}, 83(3):637--661, 2009.

\bibitem{GirPolt2010}
A.~Girouard and I.~Polterovich.
\newblock On the {H}ersch--{P}ayne--{S}chiffer estimates for the eigenvalues of
  the {S}teklov problem.
\newblock {\em Funktsional. Anal. i Prilozhen.}, 44(2):33--47, 2010.

\bibitem{Hersch1970}
J.~Hersch.
\newblock Quatre propri\'{e}t\'{e}s isop\'{e}rim\'{e}triques de membranes
  sph\'{e}riques homog\`enes.
\newblock {\em C. R. Acad. Sci. Paris S\'{e}r. A-B}, 270:A1645--A1648, 1970.

\bibitem{HPS1975}
J.~Hersch, L.~E. Payne, and M.~M. Schiffer.
\newblock Some inequalities for {S}tekloff eigenvalues.
\newblock {\em Arch. Rational Mech. Anal.}, 57:99--114, 1975.

\bibitem{JLNNP2005}
D.~Jakobson, M.~Levitin, N.~Nadirashvili, N.~Nigam, and I.~Polterovich.
\newblock How large can the first eigenvalue be on a surface of genus two?
\newblock {\em Int. Math. Res. Not.}, 63:3967--3985, 2005.

\bibitem{JakNadPoltKlein}
D.~Jakobson, N.~Nadirashvili, and I.~Polterovich.
\newblock Extremal metric for the first eigenvalue on a {K}lein bottle.
\newblock {\em Canad. J. Math.}, 58(2):381--400, 2006.

\bibitem{KarpukhinYangYau}
M~Karpukhin.
\newblock On the {Y}ang-{Y}au inequality for the first {L}aplace eigenvalue.
\newblock arXiv:1902.03473.

\bibitem{KNPP2019}
M.~Karpukhin, N.~Nadirashvili, A.~Penskoi, and I.~Polterovich.
\newblock An isoperimetric inequality for {L}aplace eigenvalue on the sphere.
\newblock {\em J. Differential Geom.}, to appear.

\bibitem{Lau2019}
R.~S. Laugesen.
\newblock The {R}obin {L}aplacian --- spectral conjectures, rectangular
  theorems.
\newblock arXiv:1905.07658.

\bibitem{LiYau}
P.~Li and S.-T. Yau.
\newblock A new conformal invariant and its applications to the {W}illmore
  conjecture and the first eigenvalue of compact surfaces.
\newblock {\em Invent. Math.}, 69(2):269--291, 1982.

\bibitem{NadTorus}
N.~Nadirashvili.
\newblock Berger's isoperimetric problem and minimal immersions of surfaces.
\newblock {\em Geom. Funct. Anal.}, 6(5):877--897, 1996.

\bibitem{NadSphere}
N.~Nadirashvili.
\newblock Isoperimetric inequality for the second eigenvalue of a sphere.
\newblock {\em J. Differential Geom.}, 61(2):335--340, 2002.

\bibitem{NadSire2017}
N.~Nadirashvili and Y.~Sire.
\newblock Isoperimetric inequality for the third eigenvalue of the
  {L}aplace-{B}eltrami operator on {$\Bbb{S}^2$}.
\newblock {\em J. Differential Geom.}, 107(3):561--571, 2017.

\bibitem{NayataniShoda}
S.~Nayatani and T.~Shoda.
\newblock Metrics on a closed surface of genus two which maximize the first
  eigenvalue of the {L}aplacian.
\newblock {\em C. R. Math. Acad. Sci. Paris}, 357(1):84--98, 2019.

\bibitem{Petrides2014}
R.~Petrides.
\newblock Maximization of the second conformal eigenvalue of spheres.
\newblock {\em Proc. Amer. Math. Soc.}, 142(7):2385--2394, 2014.

\bibitem{Szego1954}
G.~Szeg{\H o}.
\newblock Inequalities for certain eigenvalues of a membrane of given area.
\newblock {\em J. Rational Mech. Anal.}, 3:343--356, 1954.

\bibitem{Weinberger1956}
H.~F. Weinberger.
\newblock An isoperimetric inequality for the {$N$}-dimensional free membrane
  problem.
\newblock {\em J. Rational Mech. Anal.}, 5:633--636, 1956.

\bibitem{Weinstock1954}
R.~Weinstock.
\newblock Inequalities for a classical eigenvalue problem.
\newblock {\em J. Rational Mech. Anal.}, 3:745--753, 1954.

\bibitem{YangYau}
P.-C. Yang and S.-T Yau.
\newblock Eigenvalues of the {L}aplacian of compact {R}iemann surfaces and
  minimal submanifolds.
\newblock {\em Ann. Scuola Norm. Sup. Pisa Cl. Sci. (4)}, 7(1):55--63, 1980.

\end{thebibliography}

\end{document}